\documentclass[12pt,reqno,oneside]{amsart}

\usepackage{amsmath,amsthm,amssymb}
\usepackage{graphicx,xcolor}

\usepackage{etoolbox}
\numberwithin{figure}{section}
\numberwithin{equation}{section}
 
\DeclareFontFamily{U}{mathb}{\hyphenchar\font45}
\DeclareFontShape{U}{mathb}{m}{n}{
	<-6> mathb5 <6-7> mathb6 <7-8> mathb7
	<8-9> mathb8 <9-10> mathb9
	<10-12> mathb10 <12-> mathb12
}{}
\DeclareSymbolFont{mathb}{U}{mathb}{m}{n}
\DeclareMathSymbol{\llcurly}{\mathrel}{mathb}{"CE}
\DeclareMathSymbol{\ggcurly}{\mathrel}{mathb}{"CF}

\definecolor{Red}{cmyk}{0,1,1,0}

\definecolor{Blue}{cmyk}{1,1,0,0}

\theoremstyle{plain}
\newtheorem{maintheorem}{Theorem} 

\newtheorem{maincorollary}[maintheorem]{Corollary}
\newtheorem{theorem}{Theorem }[section]
\newtheorem{proposition}[theorem]{Proposition}
\newtheorem{lemma}[theorem]{Lemma}
\newtheorem{corollary}[theorem]{Corollary}

\theoremstyle{definition} \theoremstyle{remark}
\newtheorem{remark}[theorem]{Remark}
\newtheorem{example}[theorem]{Example}
\newtheorem{definition}[theorem]{Definition}

\newcommand{\diam}{\operatorname{diam}}

\renewcommand{\ge}{\geqslant}
\renewcommand{\geq}{\geqslant}
\renewcommand{\leq}{\leqslant}
\renewcommand{\le}{\leqslant}

\usepackage{mathrsfs}
\usepackage[colorlinks=true,linkcolor=blue, urlcolor=red, citecolor=blue]{hyperref}



\begin{document}
\pagestyle{myheadings}

\title[Foundations of local iterated function systems]{Foundations of local iterated function systems}

\author{Elismar R. Oliveira and Paulo Varandas}

\address{Elismar R. Oliveira,  Departamento de Matem\'atica, Universidade Federal do Rio grande do Sul\\
Porto Alegre, Brazil}
\email{elismar.oliveira@ufrgs.br}

\address{Paulo Varandas,  Center for Research and Development in Mathematics and Applications (CIDMA), 
Department of Mathematics, University of Aveiro, 3810-193 Aveiro, Portugal
\& Departamento de Matem\'atica, Universidade Federal da Bahia\\
Av. Ademar de Barros s/n, 40170-110 Salvador, Brazil}
\email{paulo.varandas@ua.pt}

\begin{abstract}
In this paper we present a systematic study of continuous local iterated function systems. We prove local iterated function systems admit compact attractors and, under a contractivity assumption, construct their code space and present an extended shift that describes admissible compositions. In particular, the possible combinatorial structure of a local iterated function system is in bijection with the space of invariant subsets of the full shift. Nevertheless, these objects reveal a degree of unexpectedness relative to the classical framework, as we build examples of local iterated function systems which are not modeled by subshifts of finite type and give rise to non self-similar attractors. 
We also prove that all attractors of graph-directed IFSs are obtained from local IFSs on an enriched compact metric space.
  
We provide several classes of examples illustrating the scope of our results, emphasizing both their contrasts and connections with the classical theory of iterated function systems.
  
\end{abstract}

\date{\today}
\maketitle

\tableofcontents

\section{Introduction}
\medskip

Iterated Function Systems (IFSs) exhibit rich topological properties that are central to their mathematical and practical significance. Defined by a collection of continuous self maps $(f_i)_{1\le i\le n}$ on a compact metric space $X$, it defines a self-similar set, typically the unique non-empty compact set $\Lambda$ that remains invariant under the action of 
the {Hutchinson-Barnsley operator} $F:  2^{X} \to 2^{X} $, defined by
$$
F(B)=\bigcup_{1\le j \le n} f_{j}(B),
\quad \text{for each $B \in 2^X$.}
$$
In case the maps $f_i$ are contractive   
it is known that the Hutchinson-Barnsley operator is a contraction, and that all sets $B\subset X$ are attracted to $\Lambda$, 
and the attractor of the IFS possesses a fractal structure. 
Notably, in case the iterated function system 
the contracting effect under iteration by the maps $f_j$ ensures that each point in $\Lambda$ is in correspondence  with an itinerary $\underline b=(b_n)_{n\le -1}\in \{1,2,\dots, n\}^{-\mathbb N}$, refered to as the code space. This is in fact a bijection whenever the IFS satisfies the open set condition, namely
\begin{equation}
    \tag{OSC}
f_i(X)\cap f_j(X)=\emptyset \quad \text{for every distinct } 1\le i,j\le n.
\end{equation}
    Overall, 
IFSs possess rich dynamical and geometric properties that have been thoroughly investigated in recent decades, leading to many significant applications both in the pure and applied sciences.
Building on Barnsley’s foundations, contemporary research extends the theory of IFSs towards probabilistic and multifractal frameworks, measure rigidity, and symbolic dynamics. 
 
In particular, the interplay with thermodynamic formalism - through concepts such as topological pressure and Gibbs measures - links geometric scaling with entropy and multifractal spectra, offering a bridge between dynamical systems and geometric measure theory. In this way, Barnsley’s operator-theoretic approach to construction of fractals remains a cornerstone in the modern understanding of fractal geo\-metry, as it furnishes a formal framework through which the principles of self-similarity, iteration, and invariance are unified under a single mathematical structure.
Our main objective is to extend Barnsley’s perspective by offering a unified theoretical framework for local iterated function systems, which we proceed to recall.

\smallskip
 
   Local Iterated Function Systems (local IFSs), written as $R_{\mathfrak X}=(X_j,f_j)_{1\le j\le n}$ for some collection of subsets $\mathfrak X =(X_j)_{1\le j \le n}$ of $X$ and to be defined precisely in Definition~\ref{def:local IFS}, were introduced by Barnsley \cite{Bar86} motivated by applications to fractal image compression and  computational mathe\-matics.
      The general picture for local IFSs  differs substantially from the classical context of IFSs. It often occurs that certain compositions of maps $f_i\circ f_j$ are not possible, making place dependent the space of admissible compositions. 
   Despite the effort and several applications and contributions in the subsequent years, an axiomatic description of the local IFSs seems far from complete. In fact, Fisher \cite{Fis95} emphasized that a complete theory for partitioning iterated function systems - nowadays known as local iterated function systems - was yet to be developed and pinpointed that in general one could not obtain the chaos game nor the contraction of the Hutchinson-Barnsley operator. This lead to the development of the theory of local iterated function systems satisfying  additional very restrictive conditions
   (cf. Remark~\ref{rmk:restrictivec} for the precise statements).
Such an approach effectively acts as a bypass to the fundamental obstacles pointed out by Barnsley, Fisher, and their collaborators, notably on the lack of a general fixed-point principle and the breakdown of global contractivity of the Hutchinson-Barnsley operator. Consequently, the foundations of a general theory remained largely unsettled.

   \smallskip
In this work, we establish the basic framework for a general theory of local iterated function systems, initiating the analysis of their topological, combinatorial, and geometric structure in the setting of a continuous local
 IFS  $R_{\mathfrak X}=(X_j,f_j)_{1\le j \le n}$ on a compact metric space $X$
 and closed subsets $X_j\subset X$.   
At first glance, one could imagine  that the topological properties of local IFSs closely resemble those of classical IFSs, with more similarities than differences. However this turns out not to be the case. A major drawback with respect to the classical theory of IFSs is that the local Hutchinson-Barnsley operator may not be a contraction. 
  Yet, it is not hard to check that the iterates of the local Hutchinson-Barnsley operator $F_{\mathfrak X}$ converge to a compact and $F_{\mathfrak X}$-invariant set $A_{\mathfrak X}\subseteq X$ in the Hausdorff-Pompeiu distance
   (cf. Theorem~\ref{thm:topological}).
    However, such local Hutchinson-Barnsley operators $F_{\mathfrak X}$
   have major differences in comparison to the classical operators. In fact, such operators may fail to be contractions whenever the maps $f_j$, and this is evidenced by the fact that there may exist subsets $B\subsetneq X$ that are not attracted to $A_{\mathfrak X}$, leading us to introduce several notions of topological basin of attractor for the local attractor of a local IFS (see Theorem~\ref{thm:topological}). 
In Theorem~\ref{thm:classification} we establish that every locally contractive IFS is semiconjugate to a locally contractive IFS on a shift space. This connection opens the way to constructing attractors of local IFSs both through geometric examples and symbolic–combinatorial techniques, enriching the understanding of their structure.

A second striking difference between the classical theory of IFSs with the one of local IFSs concerns the existence of local attractors with components that do not belong to the union of the domains of the maps that define the local IFS, leading to the notion of endpoints. More precisely, there exist contractive local IFSs 
   $R_{\mathfrak X}=(X_j,f_j)_{1\le j \le n}$, where $X_i\subset X$ are compact restrictions of the domains of the maps $f_i$ of contractive IFSs
   $R=(X,f_j)_{1\le j \le n}$ 
   whose attractors contain points that do not belong to the set
   $\bigcup_{1\le j \le n} X_j$
   (in particular such points have finite orbits and 
   the chaos game necessarily fails, see Example~\ref{ex:Elismar}). 
This motivates the definition of two possibly distinct classes of attractors
\begin{equation}\label{eq:apresentacaoA}
A^\infty_{\mathfrak X} \subseteq A_{\mathfrak X} \subseteq A_R,    
\end{equation}
where $A_{\mathfrak X}$ will be denoted as the local attractor of the local IFS and $A^\infty_{\mathfrak X}$ will stand for the part of the local attractor which has no endpoints, 
with the following properties:

{\small 
\begin{center}
\begin{tabular}{|c|c|c|c|}
\hline
 & {\bf  Code space} & {\bf  Orbits} & {\bf  Transitive} \\
\hline
$A^\infty_{\mathfrak X}$   &  { Compact shift invariant subset} & All points have infinite orbits &  No \\
\hline
$A_{\mathfrak X}$ & Compact shift invariant subset & There may exist points with finite orbits & No  \\
\hline
$A_R$ & Full shift & All points have infinite orbits & Yes \\
\hline
\end{tabular}    
\end{center}
}
\medskip
In fact, the local attractor $A_{\mathfrak X}$ of a local contrating IFS presents the following characteristics in comparison with the classical theory of IFS:
\begin{itemize}
  
    \item[$\circ$] there exist positively invariant subsets that are not attracted to $A_{\mathfrak X}$   (cf. Example~\ref{ex:differnt-basins})
    
\item[$\circ$] $A_{\mathfrak X}$ may be non-transitive and not satisfy the chaos game  (cf. Example~\ref{ex:Elismar})
\item[$\circ$] $A_{\mathfrak X}$ may have finite orbits of different length, hence endpoints (cf. Example~\ref{ex:Elismar}); 
   \item[$\circ$] $A_{\mathfrak X}$ may contain points in the complement of $\bigcup_{i} X_i$  (cf. Example~\ref{ex:shift}) 
    
    \item[$\circ$] the compact space of $A^\infty_{\mathfrak X}$ may be accumulated by $A_{\mathfrak X} \setminus A^\infty_{\mathfrak X}$ (cf. Example~\ref{ex:shift});
    \item[$\circ$] the code space of $A_{\mathfrak X}$ may not be a subshift of finite type (cf. Example~\ref{ex:nonSFT});
    \item[$\circ$] $A_{\mathfrak X}$ may be non self-similar (cf. Example~\ref{ex:MapleSierpinski}).
\end{itemize}

\begin{figure}[htb]
    \centering
  \includegraphics[width=0.4\linewidth]{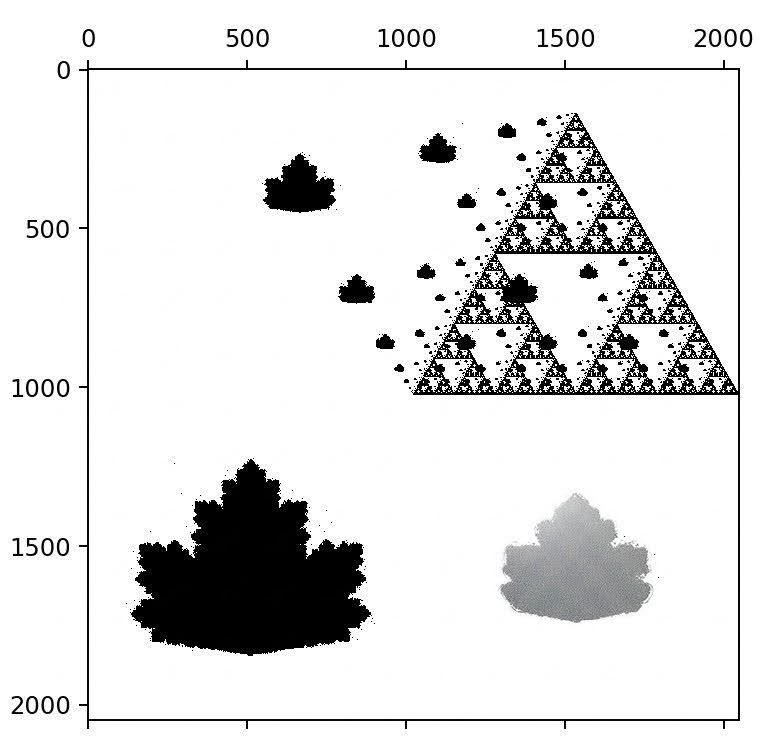}
    \caption{Local attractor $A_{\mathfrak X}$ that is non-transitive and not self-similar, in which the endpoints are marked in grey and $A^\infty_{\mathfrak X}$ is marked in black (cf. Example~\ref{ex:MapleSierpinski})}
    \label{fig:placeholder}
\end{figure}

  \smallskip
 Proceeeding with the comparison between the classical framework of IFSs and local IFSs,
 a third fundamental distinction 
 becomes apparent 
   
 in the description of the code space and dynamical features of the local attractors. In case of contractive IFSs the code space is a full shift and all points have infinite orbits by arbitrary concatenations of maps $f_j$. In particular, the dynamical properties can be coded by a skew-product (see e.g. Section~\ref{sec:skewgen}).  
This plays a fundamental role not only of the topological description of the attractor but, most importantly, on the development of a thermodynamic formalism that allows one to describe the geometric aspects of the attractors including the computation and study of regularity of their Hausdorff dimension
and their stationary measures (see e.g. \cite{FU09,Fen23a,HK25,MU96}  and references therein).

In case of contractive local IFSs, the space of orbits of a local IFS is in correspondance with the space of all admissible compositions of the maps $f_j$. In consequence,  the code space of a local IFS is identified with a (possibly proper) invariant subset of the shift space. The formulation of theses correspondance involves some intricate constructions of invariant subsets of the one-sided and two-sided shifts  
(cf. Theorem~\ref{thmC}, Theorem~\ref{thm:D} and Corollary~\ref{cor: thmD} for the precise statements).
In particular, the previous characterization allowed us to describe an essential part of the Hutchinson-Barnsley fractal operator, on which one still obtains contraction and convergence to the local attractor in the Hausdorff metric (see Theorem~\ref{thmB} for the precise statement).

Local IFSs can also be compared to some generalizations of the concept of IFS, namely  graph-directed IFSs and inhomogeneous IFSs, introduced in Mauldin and Williams \cite{MW88} and Barnsley and Demko \cite{BarnsleyDemko1985}, respectively.
In respect to the first class, we prove that every graph-directed IFS 
on a compact metric space $X$ can be realized as a contractive local IFS on some enriched compact metric space (cf. Theorem~\ref{thm:graphdirected}
and Remark~\ref{rmk:comparisongraph}).

As for the second class, inhomogeneous IFSs offer a generalization of the concept of IFS where an additional fixed set $C$ (named condensation set) is considered, and one describes the attractor $A_C\subseteq X$ which is a fixed point of the operator 
$$
F_C(B)=\bigcup_{1\le j \le n} f_{j}(B) \; \cup \; C,
\quad \text{for each $B \in 2^X$.}
$$
In general, as $F_C(B)\supset F(B)$ for every $B\subset X$ the attractor $A_C$ always contains the global attractor $A_R$ of the original IFS. Moreover, it is not hard to check that $A_C=A_R$ if and only if $C$ is contained in the global attractor $A_R$.
As the local attractors are always contained in the global attractors $A_R$, these cannot be realized as non-trivial inhomogeneous attractors.

\smallskip
Building upon the contributions that form the basis of the theory of local IFSs, we anticipate that this framework can serve as a foundation for further extensions concerning the study of their ergodic and dimensional properties, possibly extending \cite{Klo22,Fen23b,MU96} to this more general context. 
A first development is the study of the topological stability and shadowing properties for local IFSs and, in particular, establish the topological stability of contractive IFSs.
In this general framework this gives rise to a new notion, namely the concept of combinatorial stability. 
We refer the reader to \cite{OV26} for the precise definitions and statements.
We also expect that the results here can be used in developments on the relaxation of the contractivity conditions, the development of a theory of local Fuzzy IFSs (the existence of the fuzzy local attractor will appear in \cite{OSV26}) and to consider wider classes of maps, thereby broadening the scope and applicability of the theory to a wider class of dynamical systems
as obtained in \cite{LKN25,Mic20,OS18,Ste99} for IFSs (see also references therein), as well as applications to the study of foliations and their associated holonomy pseudogroups, which are natural examples of local IFSs (cf. \cite{GLW88,Wal04}).

\medskip
The paper is organized as follows. In Section~\ref{sec:mainresults} we define the basic notions related to the topology and geometry of attractors for local IFSs and state the main results. 
The existence of a compact attractor for continuous local IFSs   
is proved in Section~\ref{sec:proofs1}. 
The combinatorial description of the code space of the attractor and the space of admissible compositions is given in Section~\ref{sec:codespace}. 
In Sections~\ref{sec:orbits} and ~\ref{sec:skewtype} we construct enlarged shift spaces describing the space of orbits and admissible compositions of a local IFS on a local attractor $A_{\mathfrak X}$ and on the (possibly proper) component $A^\infty_{\mathfrak X}$ of the attractor formed by points with infinite orbits. 
In Section~\ref{sec:exponentialbasin} we establish an exponential approximation of the attractor by certain essential subsets in the iterates of the Hutchinson-Barnsley operator.
As the main results in the paper have a wide range of applications, we discuss  a series of examples at Section~\ref{sec:examples}.

\section{Main results}\label{sec:mainresults}
This section is devoted to the statement of our main results on local iterated function systems.  In Subsection~\ref{subsec:general-statements} we establish the existence of topological attractors for general local IFSs and study their basin of attraction.

\subsection{Local iterated function systems}
Let us recall the concept of local iterated function systems,  which can be traced back to \cite{BH93,Fis95,BHM14}.
Let $(X,d)$ be a compact metric space and let $ 2^{X} $ denote the hyperspace of subsets of $X$.  

\begin{definition}\label{def:local IFS}
   A \emph{local IFS} 
   on $X$ is a structure $R_{\mathfrak X}=(X_j, f_{j})_{1\le j \le n}$ where 
  $n \geq 2$ is an integer, $X_1, \ldots, X_n \in 2^{X}$ is an arbitrary family of closed subsets of $X$, and $f_{j}: X_{j} \to X$ is a continuous map for each $1\le j\le n$. We shall refer to the family $\mathfrak X=(X_j)_{1\le j \le n}$ as the set of restrictions of the local IFS.
\end{definition}

The special case in which $X_j=X$ for each $\le j \le n$ corresponds to the usual notion of iterated function system (IFS). In particular the notion of local IFS not only generalizes the notion of IFS 
as it will allow a wide range of new phenomena and attractors which may not be self-similar (see Section~\ref{sec:examples}). 
Since the continuous maps on metric spaces $f_j: X_j \to X$ admit a continuous extension $\tilde f_j : X \to X$
we will adopt the perspective that the relevant invariant sets for a local iterated function system are subsets of an attractor for an IFS.  

\begin{definition}\label{def:loc fract operat}
   The \emph{local Hutchinson-Barnsley operator}\index{Local Hutchinson-Barnsley operator} $F_{\mathfrak X}:  2^{X} \to 2^{X} $ is given by
$$F_{\mathfrak X}(B)=\bigcup_{1\le j \le n} f_{j}(B \cap X_{j}),$$
for each $B \in 2^X$. 
A set $B\subset X$ is called \emph{$F_{\mathfrak X}$-positively invariant}\index{$F_{\mathfrak X}$-positively invariant} if $F_{\mathfrak X}(B) \subseteq B$. 
If $F_{\mathfrak X}(B)=B$ then we say that $B$ is a 
\emph{$F_{\mathfrak X}$-invariant set}\index{$F_{\mathfrak X}$-invariant set}.
\end{definition}

We denote by
\begin{equation}\label{eq:definvFloc}
\operatorname{Fix}(F_{\mathfrak X})=\{B \in  2^{X} : F_{\mathfrak X}(B)=B\}    
\end{equation}
the set of all invariant sets of $R_{\mathfrak X}$.
It is clear from the definition that a local IFS could a priori admit several invariant sets, and that not all of them would eventually have the property of attracting some open set of points.
 
One can prove that, under the previous assumptions, there exists a unique $F_{\mathfrak X}$-invariant set  $A_{\mathfrak X}\subseteq X$ in which is maximal with respect to the inclusion (cf. 
Proposition~\ref{prop: set property local attract}).
We emphasize that the closedness of the elements in $\mathfrak X=(X_j)_{1\le j \le n}$  cannot be removed (cf. Example~\ref{exam:filip}).
  
Furthermore, in such degree of generality it may occur that $\operatorname{Fix}(F_{\mathfrak X})=\{\emptyset\}$ and that $A_{\mathfrak X}=\emptyset$ (e.g. in the trivial case that $X_j=\emptyset$ for all $1\le j \le n$). This motivates the following notion.

\begin{definition}
    \label{def: the local attractor}
    Let $R_{\mathfrak X}=(X_j, f_{j})_{1\le j \le n}$ be a local IFS. If $A_{\mathfrak X}\neq\emptyset$ we will refer to $A_{\mathfrak X}$  as the \emph{local attractor for $R_{\mathfrak X}$.}\index{Local attractor! For $R_{\mathfrak X}$.}
\end{definition}

In general the equality $A_{\mathfrak X}= \bigcup_{B \in \operatorname{Fix}(F_{\mathfrak X})} B$ (see the proof of 
Proposition~\ref{prop: set property local attract}) 
does not say much on the structure and properties of the set $A_{\mathfrak X}$. Thus, the study of the geometric properties of the local attractor and its dependence with respect to the domains $X_j$ and the maps $f_j$ will rely on further assumptions on the regularity of the latter. 

\subsection{Topological aspects}
 
\label{subsec:general-statements}

Our first result, whose proof will be given in Section~\ref{sec:proofs1},  ensures the existence and characterization of the maximal $F_{\mathfrak X}$-invariant compact set $A_{\mathfrak X}$.

Recall that, given a metric space $(X,d)$, the \emph{Hausdorff-Pompeiu distance}\index{Hausdorff-Pompeiu distance} on 
$K^*(X)$ is the metric  $\text{dist}_H: K^*(X) \times K^*(X) \to \mathbb{R}$  defined by
$$\text{dist}_H(A,B)=\max\{d(A,B), d(B,A)\}=\max\{\max_{a\in A}d(a,B), \max_{b\in B}d(b,A)\},$$
for all $A,B \in K^*(X)$ (here $K^*(X)$ denotes the family of non-empty compact subsets of $X$).
In the course of our studies, it was natural to consider the following extensions of the classical notion of topological basin of attraction for IFSs.

\begin{definition}\label{def: new basin of attraction}
Consider the family of positively invariant compact sets 
   $$
   \mathcal{B}_{\mathfrak X}^{\, \text{inv}} =\Big\{A_0 \in K^*(X) \; : \;  F_{\mathfrak X}(A_0)\subseteq A_0 \quad\text{and}\quad A_{\mathfrak X} \subseteq A_0\Big\},
   $$
the \emph{nested basin of attraction of $A_{\mathfrak X}$}\index{Nested basin of attraction of $A_{\mathfrak X}$} defined as 
   $$
   \mathcal{B}_{\mathfrak X}^{\, \text{out}} =\Big\{A_0 \in K^*(X) \; : \; \quad  \bigcap_{i\geq 0} F_{\mathfrak X}^{i}(A_0) = A_{\mathfrak X}\Big\},
   $$
and the \emph{topological basin of attraction of $A_{\mathfrak X}$}\index{Basin of attraction of $A_{\mathfrak X}$} as the family of sets
   $$
   \mathcal{B}_{\mathfrak X} =\Big\{A_0 \in K^*(X) \; : \;   \lim_{i \to \infty} F_{\mathfrak X}^{i}(A_0) = A_{\mathfrak X}\Big\}.
   $$
\end{definition}

The next theorem provides a complete picture of the possible compact sets attracted to the local attractor of a contractive local IFS.

\begin{maintheorem}\label{thm:topological}
Let $R_{\mathfrak X}=(X_j, f_{j})_{1\le j \le n}$ be a local IFS with local attractor $A_{\mathfrak X}$.
The following properties hold:
   \begin{enumerate}
     \item if $A_0 \supseteq A_{\mathfrak X}$ is a $F_{\mathfrak X}$-positively invariant {closed set} then 
     $$A_{\mathfrak X}= \bigcap_{i\geq 0}  F_{\mathfrak X}^{i}(A_0);$$
     \item  
     $A_{\mathfrak X}$ is a compact set  and 
          $A_{\mathfrak X}= \lim_{i\to\infty}  F_{\mathfrak X}^{i}(X),$ 
         with respect to  the Hausdorff-Pompeiu distance;
         \item 
    $\emptyset \neq \mathcal{B}_{\mathfrak X}^{\, \text{inv}} \subseteq \mathcal{B}_{\mathfrak X}^{\, \text{out}} \subseteq \mathcal{B}_{\mathfrak X}\subseteq K^*(X).$ 
   \end{enumerate}
\end{maintheorem}

Some comments are in order. In case of IFSs, one knows that $\mathcal{B}_{\mathfrak X}= K^*(X)$, hence in this case all the inclusions in Theorem~\ref{thm:topological} item (3) are strict. 
To the best of our knowledge, this is the first time the problem of defining a basin of attraction 
is addressed in the literature (we refer the reader to Example~\ref{ex:differnt-basins} for an explicit example where these notions  $\mathcal{B}_{\mathfrak X}^{\, \text{inv}}, \mathcal{B}_{\mathfrak X}^{\, \text{out}}$ and $\mathcal{B}_{\mathfrak X}$ differ).
In fact, it will be proven below (cf. Proposition~\ref{prop:set prop F_loc}) that every { $F_{\mathfrak X}$-positively invariant closed} subset $A_0 \supseteq A_{\mathfrak X}$
is contained in the basin of attraction of $A_{\mathfrak X}$, and that the latter may fail if the $F_{\mathfrak X}$-invariance condition is removed (cf. Example~\ref{ex:differnt-basins}).

In particular, in the contractive setting $A_{\mathfrak X}, A_R, X \in \mathcal{B}_{\mathfrak X}$, so the latter it is never an empty set. A less immediate property is that, in case the global IFS admits an attractor then it will also belong to the basin of attraction of any local attractor obtained form it (see Corollary~\ref{cor: iterative attractor full attractor} for the precise statement). 

Finally, one could expect that 
$A_{\mathfrak X}\subset \bigcup_{i=1}^n X_i$, however this is not the case (cf. Example~\ref{ex:Elismar}).
The previous result motivates a first natural question about the continuity of the local attractor. 

Item (2) in Theorem~\ref{thm:topological} guarantees that 
$$
A_{\mathfrak X}:= \lim_{i\to\infty}  F_{\mathfrak X}^{i}(X) = \bigcap_{i\ge 0} F_{\mathfrak X}^{i}(X).
$$

   \begin{remark}\label{rmk:restrictivec}
    Previous results on the existence of local attractors for contractive local IFSs, in the spirit of Theorem~\ref{thm:topological}, relied on additional restrictive conditions:
        either 
   $ 
   X=\bigcup_{1\le j \le n} X_j 
   $  
as in the cornerstone contributions \cite{Bar88,Bar06,Fis95}, or alternatively that
   $ 
   f_j(X_j)\subset X_j 
   $ for every $1\le j \le n
   $  
in more recent contributions including \cite{BHM14,GE18,Mas14,Jeb25}.
   \end{remark}

\subsection{Code space for contractive local iterated function systems}\label{subsec:contracting-statements}

Let $(X,d)$ be a complete metric space and let $(X_j)_{1\le j \le n}$ be an arbitrary family of non-empty closed subsets of $X$. In order to have a reasonable Hutchinson-Barnsley theory, we will focus our attention to the following setting, $R= \left(X, f_j \right)_{1\le j \le n}$ is a global IFS and  $R_{\mathfrak X}=(X_j, f_{j})_{1\le j \le n}$ is the associated local IFS obtained by restricting  $f_{j}$ to  each domain $X_{j}$.

\begin{definition}
    \label{def:contractiveLIFS}
    We say that 
    $R_{\mathfrak X}=(X_j, f_{j})_{1\le j \le n}$ is a \emph{contractive local IFS}\index{Contractive local IFS} if for each $1\le j \le n$ the map  $f_{j}: X_{j} \to X$  is a  contraction, that is, there exists $\lambda_j\in [0,1)$ so that 
    $ 
    d(f_j(x), f_j(y))\le \lambda_j \, d(x,y) 
    $ 
    for every $x,y\in X_j.$  
    We say that the local IFS satisfies the \emph{open set condition}\index{Open set condition} (OSC) provided that
$ 
f_i(X_i) \cap f_j(X_j) = \emptyset, \quad \text{for all } 1\le i, j \le n, \; i\neq j. 
$ 
\end{definition}
    
The previous notion generalizes the concept of contractive IFS, which corresponds to the special case that $X_{j}=X$ for every $1\le j \le n$. 

\medskip
The next main result concerns a description of the code space for a contractive local IFS. In the classical context of IFS, the code space is a full shift (cf. Section~\ref{sec:codespace}), in which case 
this symbolic dynamics allows to use the dynamical features of such shift spaces, which are very well known.  In the more general framework of local IFS, the code space can be substantially more intricate. 

\medskip

In order to prove a classification theorem, consider the space $\Sigma^-=\{1,2, \dots, n\}^{-\mathbb N}$ endowed with the metric
$$ 
d_-(\underline a, \underline b) = \exp\Big(\,-\sup\{N\ge 1\colon a_j = b_j, \;\text{for every}\; -N\le j \le -1 \}\Big)
$$ 
for every $\underline a, \underline b\in \Sigma^-$, where $\sup \emptyset =0$.  The next result shows that every contractive local IFS is semiconjugate to a contractive local (standard) IFS on a shift space.

\begin{maintheorem}\label{thm:classification}
Let $(X,d)$ be a compact metric space and $X_j\subset X$, for $1\le j \le n$. If $R_{\mathfrak X}=(X_j,f_j)_{1\le j \le n}$ is a contractive local IFS then there exists a $\sigma^-$-invariant and compact shift space $\Sigma^-_{\mathfrak X} \subset \Sigma^{-}$, a H\"older continuous surjective map $\pi_{\mathfrak X}: \Sigma^-_{\mathfrak X} \to A_{\mathfrak X}$ and 
subsets $B_j=\pi_{\mathfrak X}^{-1}(X_j)\subset \Sigma^-_{\mathfrak X}$ so that:
\begin{enumerate}
    \item $\Sigma^-_{\mathfrak X}=\{\underline b \in \Sigma^- \colon  \pi_{\mathfrak X}(\underline b) \neq \emptyset \} \subseteq \Sigma^-$;
    \item $f_j\circ \pi(\underline b)=\pi_{\mathfrak X} \circ \tau_j(\underline b)$ for every $\underline b\in B_j$.
\end{enumerate}
In particular $R_{\mathfrak X}$ is semiconjugate to the local IFS $\mathcal S_{\mathfrak B}=(B_j,\tau_j)_{1\le j \le n}$ and, if 
the contractive local IFS satisfies the open set condition then $\pi_{\mathfrak X}$ is a conjugacy.
  
\end{maintheorem}

\subsection{Generalized shift spaces}
\label{sec:skewgen}
In this section we present some of the core results in the paper, which characterize the orbits of a contractive local IFS.  Recall that the orbits of a contractive IFS $R=(X,f_i)_{1\le i \le n}$ can be described through the skew-product
\begin{equation}
    \label{eq:skewP}
    \begin{array}{rccc}
S: & X \times \Sigma & \to &  X \times  \Sigma  \\
    & (x,\underline a) & \mapsto & (f_{a_0}(x),\sigma(\underline a)),
\end{array}
\end{equation}
where $\Sigma=\{1,2, \dots, n\}^{\mathbb N_0}$, endowed with the usual distance, and $\sigma:\Sigma \to\Sigma$ denotes the usual shift. Moreover, its attractor $A_R$ is maximal invariant in the sense that it is not hard to show that it can be written as
\begin{equation}
    \label{eq:maxinv}
A_R =\Big\{x\in X \colon (f_{a_j}\circ \dots \circ f_{a_0})(x) \in \bigcup_{i=1}^n X_i, \, \text{for every}\, \underline a \in \Sigma \;  \text{and} \;j\ge 1\Big\}
\end{equation}
  
where $X_i=f_i(X)$.
The next main results, whose proofs occupy Section~\ref{sec:orbits}, offer extended shift spaces as a generalizations of the previous concepts. Let us introduce some notation.
\medskip

Given a contractive local IFS  $R_{\mathfrak X}=(X_j,f_j)_{1\le j \le n}$ let us first consider set 
\begin{equation}
    \label{eq:def-enlarged-Omega}
    \Omega =\Big\{((x_i,a_i))_{i\ge 0} \in (X\times \{1, 2, \dots, n\})^{\mathbb N} \colon x_{i+1}=f_{ a_i}(x_i) \; \text{for every}\; i\ge 0 \Big\}
\end{equation}
of points which admit infinite orbits
(it is implicitly assumed that $x_i\in X_{a_i}$ for every $i\ge 0$). 
This set is in analogy with ~\eqref{eq:maxinv}.
Given $k\ge 0$, denote by $\pi_{k,X}: \Omega \to X$ the projection given by $\pi_{k,X}(((x_i,a_i))_{i\ge 0})=x_k$. 
\smallskip

Consider also the shift map $\bar \sigma: \Omega \to \Omega$ given by
\begin{equation}
    \label{eq:extendedsigma}
\bar \sigma(((x_0,a_0),(x_1,a_1), (x_2,a_2), \dots))= ((x_1,a_1), (x_2,a_2), (x_3,a_3) \dots)
\end{equation}
and the maximal invariant subset $A^{\infty}_{\mathfrak X}$ of the attractor $A_{\mathfrak X}$ defined as 
\begin{equation}
    \label{eq:maxinvA}
    A^{\infty}_{\mathfrak X} =\big\{x\in A_{\mathfrak X} \colon \exists (a_j)_{j\ge 0} \text{ so that }  (f_{a_j}\circ \dots \circ f_{a_0})(x) \in \bigcup_{i=1}^n X_i, \, \text{for every}\, j\ge 0\big\}.
\end{equation}
The set $A^{\infty}_{\mathfrak X}$ is formed by points in the attractor $A_{\mathfrak X}$ which admit infinite orbits (this can be a proper subset of the local attractor, as illustrated in  Example~\ref{ex:Elismar}). Moreover,
by compactness of the shift space, one can write alternatively
\begin{align*}
 A^{\infty}_{\mathfrak X} 
    & =\big\{x\in A_{\mathfrak X} \colon \forall \, \ell\ge 0,\;  \exists (a_j)_{0\le j \le \ell} \text{ so that }  (f_{a_\ell}\circ \dots \circ f_{a_0})(x) \in \bigcup_{i=1}^n X_i, \, \big\} \\
    & = \bigcap_{\ell \ge 1} \; \bigcup_{(a_0, a_1, \dots, a_\ell)} (f_{a_\ell} \circ \dots \circ f_{a_0})^{-1} \big( \bigcup_{i=1}^n X_i \big)
\end{align*}
which is countable and nested sequence of closed, hence compact, subsets of $X$. Nevertheless, even though the compactness of $A^{\infty}_{\mathfrak X}$,  the set $A_{\mathfrak X} \setminus A^\infty_{\mathfrak X}$ may accumulate on $A^\infty_{\mathfrak X}$
(cf. Example~\ref{ex:shift}).

Finally, it is worth mentioning that the set $A^{\infty}_{\mathfrak X}$ defined above is in strong analogy with the global attractor $A_R$ (recall its characterization as maximal invariant set in ~\eqref{eq:maxinv}), the difference being that points in $A_R$ admit infinite orbits by all possible concatenations of maps, while points in $A^{\infty}_{\mathfrak X}$ are those that admit at least an infinite orbit.

\begin{maintheorem}
    \label{thmC} 
Let $R_{\mathfrak X}=(X_j,f_j)_{1\le j \le n}$ be a contractive local IFS with a local attractor $A_{\mathfrak X}$.
 
Then $A^\infty_{\mathfrak X}\neq \emptyset$. 
Furthermore, there exists $\widehat \Sigma_{\mathfrak X}\subset \Sigma^-_{\mathfrak X} \times \Sigma^+$ so that
\begin{enumerate}
    \item $\pi_{\mathfrak X}\circ\pi_{\Sigma^- }(\widehat \Sigma_{\mathfrak X})=A^\infty_{\mathfrak X}$;
    \item $\widehat \Sigma_{\mathfrak X}$ is invariant by the two-sided shift $\widehat \sigma$ defined by 
    $$
\widehat \sigma ( \dots , b_{-3}, b_{-2}, b_{-1}, \boxed{a_0}, a_1, a_2, a_3,\dots)
    = 
     ( \dots , b_{-2}, b_{-1}, a_0, \boxed{a_1}, a_2, a_3, a_4, \dots);
$$
    \item     for each $(\underline b,\underline a) \in \widehat\Sigma_{\mathfrak X}$ it holds that
$$
\widehat \sigma(\underline b, \underline a)=(\underline b \ast a_0, \sigma(\underline a)) = (\tau_{a_0}(\underline b), \sigma(\underline a)).
$$
\end{enumerate}
\end{maintheorem}

\begin{remark}
Under the right encoding, the sequence $( \dots , b_{-3}, b_{-2}, b_{-1}, \boxed{a_0}, a_1, a_2, a_3,\dots)$ defines the dynamics of a local IFS in the following way: given $x_0=\pi_{\mathfrak X}(\dots , b_{-3}, b_{-2}, b_{-1}) \in A_{\mathfrak X}$ then one can obtain the next iterates as follows $x_1= f_{a_0}(x_0)$, $x_2= f_{a_1}(x_1)$, and so on. In this way, we form an orbit $((x_0,a_0), (x_1,a_1),\ldots)$ equivalent to $$( \dots , b_{-3}, b_{-2}, b_{-1}, \boxed{a_0}, a_1, a_2, a_3,\dots).$$ In the next, we will shift to this representation.    
\end{remark}

The existence of infinite orbits may be a problem even for points in the local attractor
(cf. Example~\ref{ex:shift}). 
However, the converse holds as once one have a non-empty local attractor $A_{\mathfrak X}$ one can always find at least one infinite orbit (see Proposition~\ref{prop:keyorbit}).

 \medskip
Some comments are in order. 
The abstract constrained space obtained in Theorem~\ref{thmC} is due to the fact that we are working in great generality, not requiring the open set condition (hence points may have multiple codes) nor that all transitions are admissible (e.g. in case the domains of $f_i$ are not the whole space). 

The contractivity assumption in Theorem~\ref{thmC} is only used  
to show that $\pi_{\mathfrak X}$ is a function (in general it would be a set valued function). 
  
Furthermore, Theorem~\ref{thmC} applies to the local attractor of all contractive local IFSs satisfying any of the assumptions $f_i(X_i)\subset X_i$ for every $1\le i \le n$ or
    $ X=\bigcup_{i=1}^n X_i$ that have been previously considered in the literature. 
 
\begin{maintheorem}\label{thm:D}
Let $R_{\mathfrak X}=(X_j,f_j)_{1\le j \le n}$ be a contractive local IFS with  local attractor $A_{\mathfrak X}$.
   	There exists a subset
    $\Omega^{e} (\mathfrak X)
    \subset (X \times \{0,1,\ldots,n\})^{\mathbb{Z}}$ 
    such that the following hold:
	\begin{enumerate}
    \item   $\pi_{\Sigma^{-}} \Big(\Omega^{e} (\mathfrak X)\Big) = \Sigma^{-}_{\mathfrak X} $; 
		\item  $\pi_{0,X} (\Omega^{e}(\mathfrak X)) =  A_{\mathfrak X}$, in particular, $\Omega^{e} (\mathfrak X) \neq \emptyset$;
		\item  The next  diagram commutes\\	
		\begin{center}
			\begin{tabular}{ccc}
				$\Omega^{e}(\mathfrak X) $ & $\stackrel{\sigma^{-1}}{ \longrightarrow }$ & $\Omega^{e}(\mathfrak X)$ \\
				$  \pi_{-1,X} \times \pi_{\Sigma^{-}}  \downarrow$ &   & $\downarrow \pi_{1,X} \times \pi_{\Sigma^{-}}$ \\
				$ X \times \Sigma^{-} $             & $\stackrel{S}{\longrightarrow } $ &    $X \times \Sigma^{-} $ \\
			\end{tabular}
		\end{center}
		where $S(x, \underline b) = ( f_{b_{-1}}(x), \sigma^-(\underline b))$. 
       
	\end{enumerate} 
\end{maintheorem}

The precise description of the extended shift space $\Omega^{e}(\mathfrak X)$
is defined by the relations \eqref{eq:defOmegae} and 
\eqref{eq:defOmegaeu}.
  
In the special case that the maps $f_i$ that define the local IFS are invertible, Theorem~\ref{thm:D} admits a simpler formulation, similar to the one of Theorem~\ref{thmC}, whose proof is left as a simple exercise to the reader.

\begin{maincorollary}\label{cor: thmD}
    
    Let $R_{\mathfrak X}=(X_j,f_j)_{1\le j \le n}$ be a contractive local IFS with local attractor $A_{\mathfrak X}$ and let 
    $\Omega^{e} (\mathfrak X)
    \subset (X \times \{0,1,\ldots,n\})^{\mathbb{Z}}$ be given by Theorem~\ref{thm:D}. 
      
    If all maps $f_j: X_j \to f_j(X_j)$ are invertible then the following diagram commutes:			
        \begin{center}
			\begin{tabular}{ccc}
				$\Omega^{e}(\mathfrak X) $ & $\stackrel{\sigma^{-1}}{ \longrightarrow }$ & $\Omega^{e}(\mathfrak X)$ \\
				$  \pi_{0,X} \times \pi_{\Sigma^{+}}  \downarrow$ &   & $\downarrow \pi_{0,X} \times \pi_{\Sigma^{+}}$ \\
				$ X \times \Sigma^{+} $             & $\stackrel{\hat S}{\longrightarrow } $ &    $X \times \Sigma^{+} $ \\
			\end{tabular}
		\end{center}
		where $\hat S (x, \underline a) = ( f_{b_{-1}}^{-1}(x), \tau_{b_{-1}}(\underline a))$ and $\tau_{b_{-1}}(\underline a) =(b_{-1}, a_0, a_1,\ldots)$.
\end{maincorollary}

\medskip

In the context of a contractive local IFS $R_{\mathfrak X}$, one would aim at proving exponential convergence of sets in the basin of $A_{\mathfrak X}$ to the local attractor as a complement to the topological description proved in Theorem~\ref{thm:topological}. This however turns out not to be true in general, a fact that led us to introduce the following notion of essential part in the iterates of the Hutchinson-Barnsley operator.
For each compact set $K\subset X$, $k\ge 1$  we write 
$$
F_{\mathfrak X}^k(K) = \bigcup_{\underline b \in \Sigma^-} V_{[\underline b]_k} 
\quad\text{where}\quad
V_{[\underline b]_k} 
 =  f_{b_{-1}} \circ f_{b_{-2}} \circ \dots f_{b_{-k}}(X_{b_{-k}}),   
$$ 
where $ 
[\underline b]_k:= (b_{-k}, \dots, b_{-2}, b_{-1})
$ 
for each $\underline b\in \Sigma^-$ and $k\ge 1$
and $f_{b_j}(A)=f_{b_j}(A\cap X_{b_j})$ for every integer $-k \le j \le -1$ (in order not to overload notation we shall simply write $f_j$ instead of $\bar f_j$ throughout).

    \begin{definition}
    The \emph{essential part of $F_{\mathfrak X}^k(K) $}, denoted by $\rm{ess}(F_{\mathfrak X}^k(K))$, is defined by 
    $$
    \rm{ess}(F_{\mathfrak X}^k(K))=\bigcup V_{[\underline b]_k}
     \in K^*(X),
    $$
    with the union 
    taken over all sequences $\underline b\in \Sigma_{\mathfrak X}.$
     
    \end{definition}

    Roughly speaking, the essential part consists of cylinder sets that will never disappear in the construction of the local attractor. More precisely, if
    $     V_{[\underline b]_k}
    \subset \rm{ess}(F_{\mathfrak X}^k(K))$  
     then $V_{[\underline b]_n}\neq \emptyset$ for every $n\ge k$. 
In the special case of IFSs it is clear that 
     $\rm{ess}(F_{\mathfrak X}^k(K))=F_{\mathfrak X}^k(K)$ for every $k\ge 1$.
 
 We are now in a position to state the final result of this paper, namely that the essential part of the iterations under the Hutchinson-Barnsley attractor converges exponentially fast to the local attractor.

\begin{maintheorem}\label{thmB}
Assume that $R_{\mathfrak X}=(X_j, f_{j})_{1\le j \le n}$ is a contractive local IFS and suppose that $A_{\mathfrak X} \neq \emptyset$. The following properties hold:
    \begin{enumerate}
        \item 
   $$A_{\mathfrak X}=\bigcap_{i \geq 0} F_{\mathfrak X}^{i}(A_{R}) = \lim_{i \to \infty}  F_{\mathfrak X}^{i}(A_{R});$$
           \item 
    there exists $\lambda\in (0,1)$ so that, for any  $A_0 \in \mathcal{B}_{\mathfrak X}$ and any $i\ge 1$,
	$$\text{dist}_H(\rm{ess}(F_{\mathfrak X}^i(A_0)), A_{\mathfrak X}) \leq \lambda^i \text{diam}(X).$$
   where $\text{dist}_H(\cdot,\cdot)$ stands for the Hausdorff distance.
         
    \end{enumerate}
\end{maintheorem}

\begin{remark}
A fundamental property of a global IFS $R$ is that its unique fractal attractor $A_{R} \in K^*(X)$ verifies $F(A_{R})=A_{R}$ and $\lim_{k \to \infty}  F^{k}(A)=A_{R}$ for any $A \in K^*(X)$ (see e.g. \cite{Hut81,Bar86}). In particular this ensures that the global attractor can be approximated by the iteration of the Hutchinson-Barnsley operator from any initial compact set, including singletons. We notice that Theorem~\ref{thmB} is optimal, as the latter fundamental property may fail, 
 as illustrated by Example~\ref{ex:differnt-basins}.    
  
\end{remark}

It is worth mentioning that the constant $\lambda$ appearing in item (2) above is $\lambda=\max_{1\le j \le n} \lambda_j$, where $\lambda_j$ stands for the contraction constant of each map $f_j$. 

\medskip
To finish this section, in order to illustrate the wide range of applications of local IFSs we show that the latter encloses the one of graph-directed IFS as specific examples, in the sense that every graph-directed iterated function system can be realized as a local IFS.
Let us recall the concept of 
graph-directed IFS and set the notation.

Let $G=(V, E, i,t)$ be a finite directed graph, where $V=\{v_1,\dots,v_m\}$ is the set of vertices, $E \subset V^2$ is the set of oriented edges, the functions $i: E \to V$ and $t: E \to V$ represent the initial and the final vertices of an edge, respectively.
In other words, $e=(v_i,v_j) \in E$ if there exists a directed edge from $v_i$ to $v_j$, in which case 
$i((v_i,v_j)) = v_i \text{ and } t((v_i,v_j)) = v_j.$  

 A \emph{graph-directed IFS} $(G, (X_v)_{v\in V}, (f_e)_{e \in E})$ on a compact metric space $(X,d)$ consists of a finite directed graph $G=(V, E, i,t)$, a family of compact subsets $(X_v)_{v\in V}$ of $X$ and a family of contractions  
$f_e : X_{t(e)} \to X_{i(e)},   e\in E.$ 
 
To a graph-directed IFS as above one can associate the space
\begin{equation}
    \label{eqdefK}
\mathcal{K}:=\Big\{(B_v)_{v\in V} \; | \; B_v \subseteq K^*(X_v),\; \forall v \in V
\Big\}
\end{equation} 
and consider the fractal operator $F_G: \mathcal{K} \to \mathcal{K}$ given by 
\[
(F_G((B_v)_{v\in V}))_u:= \bigcup_{e\in E, \; i(e)=u} f_e(B_{t(e)}),
\]
for each $u\in V$ and  $B=(B_v)_{v\in V} \in \mathcal{K}$.
An element $A:=(A_v)_{v\in V} \in \mathcal{K}$ satisfying $F_G(A)=A.$ is called an invariant set for the graph-directed IFS $(G, (X_v)_{v\in V}, (f_e)_{e \in E})$.
It is a known that  every contractive graph-directed IFS has a unique invariant set, which is attractive with respect to the operator $F_G$, hence an attractor, satisfying
\begin{equation}
    \label{eq:GDIFSinv}
    A_v= \bigcup_{e\in E, \; i(e)=v} f_e(A_{t(e)}).
    \end{equation}
(cf. \cite[Proposition 3.4]{Verma} and references therein).

We will say that a metric space $Y$ is an enriched version of the metric space $X$ if there exists an isometric embedding $X \to Y$.
Let $(X,d)$ be a compact metric space and let
$(G, (X_v)_{v\in V}, (f_e)_{e \in E})$ 
be a contractive graph-directed IFS on $X$.
Consider the space 
$Y := X \times V$, 
equipped with the metric
$$
d_Y\bigl((x,v),(y,w)\bigr) = d(x,y) + \mathbf{1}_{\{v\neq w\}},
$$
where $\mathbf{1}_{\{v\neq w\}}=1$ if $v\neq w$ and it is equal to $0$ otherwise.
Since $V$ is finite, it follows that $(Y,d_Y)$ is a compact metric space.

We proceed to define a suitable contractive local IFS on $Y$.
For each $e\in E$, consider 
the map
\[
g_e : Y \to Y
\quad\text{given by}\quad
g_e(x,v) = \bigl(f_e(x), i(e)\bigr)
\]
It is straightforward to show that each map $g_e$ is a contraction on $Y$: as $f_e$ is a contraction with Lipschitz constant $0\le \lambda_e<1$, then for any
$(x,v),(y,w)\in D_e$,
\[
d_Y\bigl(g_e(x,v),g_e(y,w)\bigr)
= d\bigl(f_e(x),f_e(y)\bigr)
\le \lambda_e\, d(x,y)
\le \lambda_e\, d_Y\bigl((x,v),(y,w)\bigr).
\]
Now, for each $e\in E$, take the set $D_e = X_{t(e)} \times \{t(e)\} \subset Y$.

Consider the contractive local IFS $R_{D}=(D_e,g_e)_{e\in V}$, and denote by $B\subset Y$ its local attractor. In particular
\begin{equation}
\label{eq:attractB}
B=\bigcup_{e\in E} g_e(B\cap D_e).    
\end{equation}
Since $B\subset Y$ it induces an element $(B_v)_{v\in V}$ in $\mathcal{K}$ (recall 
~\eqref{eqdefK}) by
$$
B_v = \pi_1(B \cap (X \times \{v\}))
\quad\text{for each $v\in V$.}
$$

We claim that $(B_v)_{v\in V}$ is a fixed point for the fractal operator $F_G$. In fact, from \eqref{eq:attractB} and the fact that $D_e\subset X \times \{t(e)\}$, $f_e(X_{t(e)})\subset X_{i(e)}$ and
$g_e(B\cap D_e)\subset X_{i(e)}\times \{i(e)\}$, one deduces that 
\begin{equation*}
\label{eq:attractB2}
B \subset \bigcup_{e\in E} X_{i(e)} \times \{i(e)\}    
\end{equation*}
hence
$$
B_v = \pi_1(B \cap (X \times \{v\})
= \pi_1 \Big( B \cap (X_v \times \{v\})\Big)
$$
for each $v\in V$. Moreover,
\begin{align*}
B & =\bigcup_{e\in E} g_e(B\cap D_e) 
=
\bigcup_{v\in V}
\left(
\bigcup_{e:\, i(e)=v} g_e(B\cap D_e)
\right)
& = 
\bigcup_{v\in V}
\left(
\bigcup_{e:\, i(e)=v} f_e( \pi_1(B\cap D_e))
\right)
\times\{v\}
\end{align*}
where
$\pi_1: Y \to X$ denotes the projection on the first coordinate. Since 
$ 
B=\bigcup_{v\in V} B_v \times \{v\}
$ 
we conclude that 
$$
B_v = \bigcup_{e:\, i(e)=v} f_e( \pi_1(B\cap D_e))
= \bigcup_{e:\, i(e)=v} f_e( B_{t(e)})
$$
which proves ~\eqref{eq:GDIFSinv}.
Since the fractal operator $F_G$ has a unique fixed point we deduce that $(B_v)_{v\in V}$ is the attractor for the contractive graph-directed 
IFS.
Altogether, this proves the following theorem.

\begin{maintheorem}\label{thm:graphdirected}
Every attractor of a contractive graph-directed IFS 
$(G, (X_v)_{v\in V}, (f_e)_{e \in E})$ on a compact metric space $(X,d)$ is obtained from a contractive local IFS on a compact subset of an enriched compact metric space $Y$.
\end{maintheorem}

\begin{remark}
    \label{rmk:comparisongraph}
    It is clear that there exist local IFSs which are not graph-directed. In fact, 
    it is enough for the attractor of the local IFS to have endpoints or for the images of the domains not to satisfy the systems of set equations ~\eqref{eq:GDIFSinv}.
\end{remark}

\section{Existence of a local attractor and its topological basin}
\label{sec:proofs1}
In this section we prove results on local IFS generated by continuous maps, stated in Theorem~\ref{thm:topological}. The first results guarantees the existence of a maximal invariant set.

\begin{proposition}\label{prop: set property local attract}
  Let $R_{\mathfrak X}=(X_j, f_{j})_{1\le j \le n}$ be a local IFS. The set $(\operatorname{Fix}(F_{\mathfrak X}), \subseteq)$ is a non-empty and direct set. In particular, there exists a unique maximal invariant set $A_{\mathfrak X}\subseteq \operatorname{Fix}(F_{\mathfrak X})$, and it satisfies $A_{\mathfrak X}= \bigcup_{B \in \operatorname{Fix}(F_{\mathfrak X})} B$.
\end{proposition}

\begin{proof}
It is evident that $(\operatorname{Fix}(F_{\mathfrak X}), \subseteq)$ is a partially ordered non-empty set. 
   Indeed, in order to see that $\operatorname{Fix}(F_{\mathfrak X}) \neq \emptyset$ we notice that $\emptyset \in  2^{X} $ and
   $$F_{\mathfrak X}(\emptyset)=\bigcup_{1\le j \le n} f_{j}(\emptyset \cap X_{j})=\emptyset$$
   thus $\emptyset \in \operatorname{Fix}(F_{\mathfrak X})$.
   Now, given a totally  ordered sequence 
   $A_1 \subseteq A_2 \subseteq A_3 \ldots$ in $\operatorname{Fix}(F_{\mathfrak X})$, we claim that the set $B=\bigcup_{i \geq 1} A_i$ belongs to $\operatorname{Fix}(F_{\mathfrak X})$ as
\[F_{\mathfrak X} \Big(\bigcup_{i \geq 1} A_i\Big) = \bigcup_{i \geq 1} F_{\mathfrak X} (A_i) = \bigcup_{i \geq 1} A_i\]
(because $A_i \in \operatorname{Fix}(F_{\mathfrak X})$ for all $i \geq 1$), hence
it is an upper bound for the sequence. By Zorn's lemma there exists at least one maximal element $A_{\mathfrak X}$ in $(\operatorname{Fix}(F_{\mathfrak X}), \subseteq)$.
   It remains to prove that the maximal element is unique. Notice that $(\operatorname{Fix}(F_{\mathfrak X}), \subseteq)$ is a direct set:  if $M,N \in \operatorname{Fix}(F_{\mathfrak X})$ then $B= M \cup N \in \operatorname{Fix}(F_{\mathfrak X})$ because
   \begin{align*}
       F_{\mathfrak X}(M \cup N) & =\bigcup_{1\le j \le n} f_{j}((M \cup N) \cap X_{j}) \\
       & =\bigcup_{1\le j \le n} f_{j}(M \cap X_{j}) \cup \bigcup_{1\le j \le n} f_{j}(N \cap X_{j}) = M \cup N.
   \end{align*}
   This implies the uniqueness of the local attractor $A_{\mathfrak X}$. Indeed, otherwise there would exist two distinct maximal sets $A, A'$ in $(\operatorname{Fix}(F_{\mathfrak X}), \subseteq)$ from which one would produce a strictly larger maximal element $A \cup A'$, reaching a contradiction. 
   Finally, we note that one necessarily has the equality $A_{\mathfrak X}= \bigcup_{B \in \operatorname{Fix}(F_{\mathfrak X})} B$, since the set $\bigcup_{B \in \operatorname{Fix}(F_{\mathfrak X})} B$ contains all $F_{\mathfrak X}$-invariant subsets. 
\end{proof}

Let us proceed with the stufy of attracting sets for local IFSs, a fact that motivates the following notion.
\begin{definition}\label{def: sequencial local attractor}
   Suppose that for some   $A_0 \in  K^*(X)$ the property $A_{i+1} \subseteq A_{i}, \; i \geq 0$  holds, where $A_i=F_{\mathfrak X}^i(A_0)$ for every $i\ge 0$. Then $A'= \bigcap_{i\geq 0} A_{i} \subseteq A_0$ is called a \emph{sequential local attractor}\index{Sequential local attractor} for $R_{\mathfrak X}$.  
\end{definition}

In the next proposition we prove some additional properties of the local Hutchinson-Barnsley operator $F_{\mathfrak X}$, which are not immediate as this operator may even fail to be semicontinuous. Nevertheless, in 
item (2) in Proposition~\ref{prop:set prop F_loc} below ensures that every $F_{\mathfrak X}$-{positively} invariant subset $A_0 \in K^*(X)$ induces a nested sequence as above, giving rise to a sequential local attractor which can be a proper subset of $A_{\mathfrak X}$ (cf. Example~\ref{ex:differnt-basins}).

\begin{proposition}\label{prop:set prop F_loc}
   Let $R_{\mathfrak X}=(X_j, f_{j})_{1\le j \le n}$ be a local IFS.
   \begin{enumerate}
     \item  If $A \subseteq B$ then $F_{\mathfrak X}(A) \subseteq F_{\mathfrak X}(B)$ for any $A, B \in  2^{X} $.
     \item  
     If $A_0\in  2^{X}$ is $F_{\mathfrak X}$-positively invariant and $A_i=F_{\mathfrak X}^{i}(A_0)$ then $A_{i+1} \subseteq A_{i}, \; i \geq 0$. 
     \item 
     Suppose that {$A_0 \in K^*(X)$} is $F_{\mathfrak X}$-positively invariant and $A'= \bigcap_{i\geq 0} F_{\mathfrak X}^{i}(A_0)$ is its sequential attractor. If  $F_{\mathfrak X}^{i}(A_0)\neq \emptyset$ for every $i\ge 0$ 
     then $A' \in K^*(X)$  and $F_{\mathfrak X}(A')=A'$.
     \item 
     If the local IFS has a local attractor $A_{\mathfrak X}$ and $A_0$ is a closed $F_{\mathfrak X}$-positively invariant  subset that contains $A_{\mathfrak X}$ then $A'=A_{\mathfrak X}$. In particular $$A_{\mathfrak X}= \bigcap_{i\geq 0}  F_{\mathfrak X}^{i}(A_0)$$
     is a compact set.
       
     \item If $A_{\mathfrak X}\neq\emptyset$ then
         $$A_{\mathfrak X}= \bigcap_{i\geq 0}  F_{\mathfrak X}^{i}(X)=\lim_{i\to\infty}  F_{\mathfrak X}^{i}(X),$$
         with respect to  the Hausdorff-Pompeiu distance.
   \end{enumerate}
\end{proposition}
\begin{proof}
Let us prove each item separately.
The monotonicity of the local Hutchinson-Barnsley operator stated in item (1) follows because, as $A \subseteq B$ we obtain $A\cap X_j \subseteq B\cap X_j$ for all $1\le j\le n$. Thus, $f_{j}(A\cap X_j ) \subseteq f_{j}(B\cap X_j )$ for all $1\le j\le n$, which proves the 

\smallskip
The proof of item (2) is by induction. Notice that trivially
       $A_1=F_{\mathfrak X}(A_0) \subseteq  A_0$. 
      Now, assuming that $A_{i} \subseteq A_{i-1}$, by the monotonicity of the local Hutchinson-Barnsley operator and definition of the sequence $A_i$ one concludes that
      $$A_{i+1}= F_{\mathfrak X}(A_{i}) \subseteq F_{\mathfrak X}(A_{i-1})=A_i, \quad\text{for every } i\ge 0.$$

Let us now prove item (3). Using item (2), 
given any $F_{\mathfrak X}$-positively invariant
set {$A_0 \in K^*(X)$} one can consider its sequential attractor $A'= \bigcap_{i\geq 0} A_{i} \subseteq A_0$. We claim that $A'$ is non-empty as a consequence of the Cantor intersection theorem. Indeed, $A'$ is a countable intersection of compact nested sets, as $F_{\mathfrak X}$ preserves compact sets and $F_{\mathfrak X}^{i}(A_0)\neq \emptyset$ for every $i\ge 0$ by assumption. 
It remains to show that $F_{\mathfrak X}(A')=A'$. As $F_{\mathfrak X}$ may fail to be continuous, hence we cannot distribute the operator in the infinite intersection, we will prove the identity directly.

Given $x \in F_{\mathfrak X}(A')$ there exist $1\le j_0\le n$ and  $y \in A' \cap X_{j_0}$ such that $x=f_{j_0}(y)$. Since $y \in \bigcap_{i\geq 0} A_{i}$ we get $x  \in \bigcap_{i\geq 1} A_{i} = A'$, which proves that $F_{\mathfrak X}(A') \subseteq A'$.
Conversely, take an arbitrary point $x \in A'=\bigcap_{i\geq 1} A_{i}$. Notice that
for each $i\ge 1$
$$
F_{\mathfrak X}^i(A_0) = \bigcup 
  f_{j_1} \circ f_{j_2} \circ \dots f_{j_i}(X_{j_i}),   
$$ 
where the union is taken over all possible $1\le j_1, \dots, j_i \le n$. 
Hence, for each $i\ge 1$ one can write 
\[x=f_{j_{i,1}} \circ f_{j_{i,2}} \circ \dots f_{j_{i,i}}(y_i)\]
for some $y_i \in X_{j_{i,i}}$.
As the sequence $(j_{i,1})_{i\ge 1}$ takes values in the finite set  $\{1,2,\dots, n\}$,  
the piggeonhole principle ensures that there exists a subsequence along which the sequence is constant to some $1\le j_* \le n$ (for notational simplicity we will abuse notation and keep denoting this subsequence by $(j_{i,1})_{i\ge 1}$). In this way, for each $i\ge 1$,
$$
x= f_{j_*}(z_i), \quad\text{where}\quad 
    z_i=f_{j_{i,2}} \circ \dots f_{j_{i,i}}(y_i) \in X_{j_*}
$$
The closedness (hence compactness) of $X_{j_*}$ and continuity of $f_{j_*}$ ensures that there exists an accumulation point $z \in A'\cap X_{j_*}$ of the sequence $(z_i)_{i\ge 1}$ so that $x=f_{j_*}(z) \in F_{\mathfrak X}(A')$. This proves item (3).
 
\smallskip

Let us prove item (4). Now, let $A_0 \supseteq A_{\mathfrak X}$ be a $F_{\mathfrak X}$-positively invariant subset. By monotonicity of the operator, it follows that
       $F_{\mathfrak X}^{i}(A_0) \supseteq F_{\mathfrak X}^{i}(A_{\mathfrak X})=A_{\mathfrak X} \neq \emptyset$ for every $i\ge 0$ 
      hence
       $\bigcap_{i\geq 0} F_{\mathfrak X}^{i}(A_0) \supseteq A_{\mathfrak X}.$ 
      Then, by the maximality of $A_{\mathfrak X}$, we obtain $A'=A_{\mathfrak X}$.  
      Hence, from item (3) one concludes that $A_{\mathfrak X} \in K^*(X)$.

      \smallskip
Finally, the first equality in item (5) is a direct consequence of item~(4) because $A_0=X$ is obviously $F_{\mathfrak X}$-positively invariant and contains $A_{\mathfrak X}$. The second equality of item (5) relies on the description of the 
local Hutchinson-Barnsley operator. Indeed, the hypothesis ensure that the local Hutchinson-Barnsley operator satisfies $F_{\mathfrak X}( K^{*}(X))\subseteq  K^{*}(X)$ because, for $A_{i-1} \in K^{*}(X)$ and $X_j$ closed implies that $A_{i-1} \cap X_j$ is either empty or compact. Since $A_i$ must be non-empty for all $i\ge 1$ and each $f_j$ is continuous one has that $f_j(A_{i-1} \cap X_j) \in K^{*}(X)$ for some indices $j$, and possibly empty for others indices. Thus, $A_i$ is a non-empty compact set for all $i$, hence the local attractor is non-empty by item (3). The convergence follows from known properties of the Hausdorff-Pompeiu metric in $K^{*}(X)$ \cite[item (vii)]{Le03}.
This completes the proof of the proposition.
\end{proof}

\begin{remark}
The $F_{\mathfrak X}$-invariance assumption is essential in Proposition~\ref{prop:set prop F_loc}.
Given an arbitrary set $A_0 \subsetneq X$, the sequence $(A_i)_{i\ge 0}$ given by $A_i=F_{\mathfrak X}^{i}(A_0)$ is not necessarily decreasing
(e.g. in case $X_0=\{x\}$ where $f_j(x)\neq x$ for every $1\le j \le n$).
 
\end{remark}

\begin{remark}
If one removes the hypothesis that the compact set $A_0$ contains the local attractor in Proposition~\ref{prop:set prop F_loc} one can obtain proper invariant sets $A'$ strictly contained in $A_{\mathfrak X}$ (cf. Example~\ref{ex:differnt-basins}). As a consequence, we will conclude that even in the context of metric spaces in general one cannot expect the map  $F_{\mathfrak X}$ to be a contraction, with respect to  the Hausdorff-Pompeiu metric,  unless we impose some strong conditions such as all restrictions to be equal to the whole space (a global IFS).  Furthermore, we will provide examples of local attractors which are not transitive neither depends continuously on the restrictions.
\end{remark}

A very  interesting feature of local attractors, which is not existent for classical IFS, is its dependence on the domain of each map. The following proposition characterizes that as a kind of lateral upper semicontinuity. 

\begin{proposition}\label{prop: monotonicity local attractor}
	 Let $R= \left(X, f_j \right)_{1\le j \le n}$ be a global IFS, $(T,\le)$ be a totally ordered set and $T \ni t \mapsto \mathfrak{X}^t \in (K^*(X)^n)$ be a non-decreasing family of restrictions, that is, if $t \le s$ then $X_j^t \subseteq X_j^s$ for each $1\le j \le n$. Then:
	 \begin{enumerate}
	 	\item for any compact set $K$, if $t\le s$  then $F_{\mathfrak{X}^t}(K)  \subseteq  F_{\mathfrak{X}^s}(K)$;
	 	\item the function $T \ni t \mapsto A_{\mathfrak{X}^t}$ (the local attractor for the local IFS defined by  $\mathfrak{X}^t$) is non-decreasing.
	 \end{enumerate}
\end{proposition}
\begin{proof}
	 For the first item, consider any compact set $K$.  Since  $K \cap X_j^t \subseteq K \cap X_j^s$ we obtain
	 \[F_{\mathfrak{X}^t}(K)  =\bigcup_{1\le j \le n} f_{j}(K \cap X_{j}^t) \subseteq  \bigcup_{1\le j \le n} f_{j}(K \cap X_{j}^s) =F_{\mathfrak{X}^s}(K).\]
	 For the second item, we recall that, by Proposition~\ref{prop:set prop F_loc}, 
	 \[A_{\mathfrak{X}^t} = \bigcap_{i\geq 0} F_{\mathfrak{X}^t}^i(X) \subseteq \bigcap_{i\geq 0} F_{\mathfrak{X}^s}^i(X)= A_{\mathfrak{X}^s}.\]
\end{proof}

The monotonicity of local attractors appeared naturally in bifurcations of  local IFSs 
whenever one parameterizes a local IFS obtained from an IFS in terms of non-increasing families of restrictions.

\subsection{Proof of Theorem~\ref{thm:topological}}

Items (1) and (2) of Theorem~\ref{thm:topological} have been proved above. It remains to prove item (3). It is a direct consequence of Proposition~\ref{prop:set prop F_loc}. In fact,
    by Proposition~\ref{prop:set prop F_loc} we know that $X \in \mathcal{B}_{\mathfrak X}^{\, \text{inv}}$, thus  $\mathcal{B}_{\mathfrak X}^{\, \text{inv}} \neq \emptyset$. 
        Moreover, the inclusion $\mathcal{B}_{\mathfrak X}\subset K^*(X)$ hold by Definition~\ref{def: new basin of attraction}. We now prove the remaining inclusions. 

\medskip
\noindent {\bf   Claim 1:} $\mathcal{B}_{\mathfrak X}^{\, \text{inv}} \subset \mathcal{B}_{\mathfrak X}^{\, \text{out}}$
\bigskip

    Consider $A_0 \in K^*(X)$ such that $F_{\mathfrak X}(A_0)\subseteq A_0 \quad\text{and}\quad A_{\mathfrak X} \subseteq A_0$. By Proposition~\ref{prop:set prop F_loc} items (2)-(4) the sequence $(F_{\mathfrak X}^{i}(A_0))_{i\ge 0}$ is monotone decreasing to the local attractor  $A_{\mathfrak X}$.    
    Thus, $A_0 \in  \mathcal{B}_{\mathfrak X}^{\, \text{out}}$.
\bigskip

   \noindent {\bf Claim 2:} $\mathcal{B}_{\mathfrak X}^{\, \text{out}} \subset \mathcal{B}_{\mathfrak X}$.
\bigskip

    Consider $A_0 \in K^*(X)$ such that $\bigcap_{i\geq 0} F_{\mathfrak X}^{i}(A_0) = A_{\mathfrak X}$. Then, by \cite[item (vii)]{Le03}  $\lim_{i \to \infty} F_{\mathfrak X}^{i}(A_0) = A_{\mathfrak X}$. Thus, $A_0 \in  \mathcal{B}_{\mathfrak X}$.
This finishes the proof of the theorem.
\hfill $\square$

\section{Code spaces}
\label{sec:codespace}

In this section we describe the code space of a contractive local IFS, hence proving Theorem~\ref{thm:classification}. In order to establish a comparison with the classical framework, we consider first the code space of an IFS.

\subsection{Code space for contractive IFS}

Let $R= \left(X, f_j \right)_{1\le j \le n}$ be a contractive IFS on a compact metric space $(X,d)$, and let
$\Sigma=\{1,\ldots,n\}^{\mathbb{N}}$ denote the full shift space. Fix an arbitrary point $x\in X$.
On the one hand, to each sequence $\underline a=(a_{0}, a_{1}, \ldots) \in \Sigma$ one associates the orbit of the sequence of points in $X$ given by the sequence
\begin{equation}
    \label{eq-codea-orbit}
    x, f_{a_{0}}(x), f_{a_{1}}(f_{a_{0}}(x)), ...
\end{equation}
The contractivity of the maps $f_j$ ($1\le j \le n$), guarantees that the diameter of the sets $f_{a_{k}} (\dots  f_{a_{1}}(f_{a_{0}}(X)))$ tends to zero as $k$ tends to infinity and, in particular, the asymptotic behavior of the sequence ~\eqref{eq-codea-orbit} does not depend on the point $x$. 

 \medskip
On the other hand, it is natural to introduce a code space in the following way. Given $\underline b=(\ldots, b_{-2}, b_{-1}) \in \Sigma^-:=\{1,\ldots,n\}^{-\mathbb{N}}$, consider the 
\emph{code map}\index{Code map} defined by
\begin{equation}
    \label{eq:codemap-IFS}
    \pi(\underline b) = \lim_{k \to \infty} f_{b_{-1}}(\cdots (f_{b_{-k}}(X)).
\end{equation}
Recall the concatenation of elements in $\Sigma^-$ with symbols, defined by
$$
\underline b \ast j := (\ldots, b_{-3}, b_{-2}, b_{-1} )\ast j=(\ldots, b_{-3}, b_{-2}, b_{-1},j)
$$
for any $\underline b=(\ldots, b_{-3}, b_{-2}, b_{-1} )\in \Sigma^-$ and $j \in \{1,\ldots,n\}$.
It is not hard to check that $\pi$ is a continuous semiconjugacy between the original IFS and the standard contractive IFS on $\Sigma^-$: $S=(\Sigma^-, \tau_j)_{1\le j \le n}$, given by the maps $\tau_j(\underline b)= \underline b \ast j$. Indeed, it is easy to check that  the following diagram commutes 
\begin{center}
	\begin{tabular}{ccc}
	$\Sigma^-$ & $\stackrel{\tau_j }{ \longrightarrow }$ & $\Sigma^-$ \\
	$\downarrow \pi$ &   & $\downarrow \pi$ \\
	$X$             & $\stackrel{ f_j }{\longrightarrow } $ &    $X$ \\
\end{tabular}
\end{center}
that is, $f_{j}(\pi(\underline b)) = \pi(\tau_j(\underline b))$ for every $1\le j \le n,\; \underline b \in \Sigma^-$, and $\pi(\Sigma^-)=A_R$.
Hence, the space $\Sigma$ is called the \emph{code space}\index{Code space} for $R$.

\subsection{Code space for contractive local IFSs}

This subsection is devoted to the description of the code space for a contractive local IFS in Theorem~\ref{thm:classification}.
  
Let $R= \left(X, f_j \right)_{1\le j \le n}$ be a contractive IFS and  $R_{\mathfrak X}=(X_j, f_{j})_{1\le j \le n}$ be the associated local IFS, obtained by restriction of the contractive IFS to closed subsets $X_j\subset X$ ($1\le j \le n$). 
A key difficulty is that, as pointed by Barnsley \cite[p.389]{Bar06}, the shift transformation acting on the code space of local IFSs may not correspond to any finite-order Markov chain.

Inspired by the definition of the code space for IFS one aims at defining a code map and a code space by inspecting the $\omega$-limit points by taking successive orbits ending at a point $x \in A_{\mathfrak X}$. This will be done by a recursive use of the $F_{\mathfrak X}$-invariance condition 
$A_{\mathfrak X}=\bigcup_{1\le i \le n} f_i(A_{\mathfrak X} \cap X_i)$. Indeed, 
given $x \in A_{\mathfrak X}$ there exists $x_{-1} \in A_{\mathfrak X} \cap X_{b_{-1}}$ and $1\le b_{-1}\le n$ so that $x=f_{b_{-1}}(x_{-1})$. Similarly, for each $k\ge 1$ one can write
\begin{equation}
    \label{eq:preorbits-code}
x= f_{b_{-1}} \circ f_{b_{-2}} \circ \dots f_{b_{-k}}(x_{-k})
\end{equation}
where  $1\le b_{-k} \le  n$ 
and $x_{-k}\in X_{-b_k}$.

Notice that, since all the maps $f_i$ are contractions, each point $x\in A_{\mathfrak X}$ can be constructed via the process in ~\eqref{eq:preorbits-code} independently of the points $(x_{-k})_{k\ge 1}$.
Indeed, for each $\underline b\in \Sigma^-$ and $k\ge 1$, since the sequence $(V_{[\underline b]_k})_{k\ge 1}$ is decreasing and the diameter of the sets $V_{[\underline b]_k}$ tends to zero as $k \to\infty$ (recall that the maps are contractions), in the limit we will have either the empty set or a singleton. This allow us to define the local code map as follows.

\begin{definition}\label{def: code map}
   Given a local IFS we define the \emph{local code map}\index{Local code map} $\pi_{\mathfrak X}: \Sigma^- \to 2^{X} $ as the set function
   $$\pi_{\mathfrak X}(\underline b)= \bigcap_{k \geq 1} V_{[\underline b]_k} = \lim_{k\to \infty} V_{[\underline b]_k}.$$
   We define the \emph{local code space}\index{Local code space} as the set
    \begin{equation}
        \label{eq:def-localcodespace}
    \Sigma^-_{\mathfrak X}:=\{\underline b \in \Sigma^- \colon  \pi_{\mathfrak X}(\underline b) \neq \emptyset \} \subseteq \Sigma^-.
    \end{equation}
\end{definition}

\medskip
By construction, for each $\underline b\in \Sigma^-_{\mathfrak X}$ one has that $\pi_{\mathfrak X}(\underline b)$ is a single point.

\begin{lemma}
\label{le:compactcodespace}
Let $R_{\mathfrak X}=(X_j, f_{j})_{1\le j \le n}$ be a local IFS and assume that $A_{\mathfrak X}\neq \emptyset$.
	Then the space $\Sigma^-_{\mathfrak X}$ is a non-empty compact set.
\end{lemma}

\begin{proof}
Consider $\underline b ^m \in \Sigma^-_{\mathfrak X}$ a sequence converging to $\underline b ^* \in \Sigma^-$, with respect to the metric of the ambient space. We claim that $\underline b ^* \in \Sigma^-_{\mathfrak X}$. To see that it is enough to prove that $\pi_{\mathfrak X}(\underline b)$ is a non-empty set. We know that if $d(\underline b ^m ,  \underline b ^*)< e^{-N}$ then 
$b ^m_{-i} =  \underline b ^*_{-i}$ for $-N \leq -i \leq -1$. This means that 
\[V_{[\underline b^*]_N} = V_{[\underline b^m]_N} \neq \emptyset.\]
As the intersection of any nested sequence of compact sets  is not empty, one has proved that $\pi_{\mathfrak X}(\underline b^*)= \bigcap_{k \geq 1} V_{[\underline b^*]_k}  \neq \emptyset$. Meaning that $\underline b ^* \in \Sigma^-_{\mathfrak X}$. 
 This proves that $\Sigma^-_{\mathfrak X}$ is a closed subset of the compact set $\Sigma^-$, hence compact, and that $\Sigma^-_{\mathfrak X}\neq \emptyset$, as desired.   
\end{proof}

Now, let us consider the \emph{address map}\index{Address map} 
$\Psi: X \to 2^{\{1,\ldots,n\}}$ given by
$$
\;\Psi(x)=\{1\le j\le n \colon x\in X_j \}.
$$
It is not hard to check that in the special context of IFS (i.e. $X_j=X$ for all $1\le j \le n$) one has that the code space $\Sigma^-_{\mathfrak X}=\Sigma^-$ is the full shift, as $\Psi(x)=\{1,2,\dots, n\}$ for every $x\in X$.

\begin{proposition}\label{prop: properties code map}
    The following properties hold:
    \begin{enumerate}
      \item For each $x \in A_{\mathfrak X}$ there exists $\underline b \in \Sigma^-_{\mathfrak X}$ so that $\pi_{\mathfrak X}(\underline b)=\{x\}$;
      \item       If $\underline b\in \Sigma^-_{\mathfrak X}$ and $j \in \Psi(\pi_{\mathfrak X}(\underline b))$ then $\underline b \ast j\in \Sigma^-_{\mathfrak X}$ and $\pi_{\mathfrak X}(\underline b \ast j)=f_{j}(\pi_{\mathfrak X}(\underline b))$;
      \item  $A_{\mathfrak X} = \pi_{\mathfrak X}(\Sigma^-_{\mathfrak X})$;
      \item the shift $\sigma^-: \Sigma^- \to \Sigma^-$ given by 
    $\sigma^-((b_{k})_{k\le -1})= (b_{k-1})_{k\le -1}$ verifies $\sigma^{-}(\Sigma^-_{\mathfrak X})=\Sigma^-_{\mathfrak X}$.
    \end{enumerate}
\end{proposition}
\begin{proof}
Item (1) follows from the definition of $\Sigma^-_{\mathfrak X}$, and it ensures that 
$A_{\mathfrak X} \subseteq \pi_{\mathfrak X}(\Sigma^-_{\mathfrak X})$.  In fact, using recursively 
the $F_{\mathfrak X}$-invariance condition 
$A_{\mathfrak X}=\bigcup_{1\le i \le n} f_i(A_{\mathfrak X} \cap X_i)$ as in ~\eqref{eq:preorbits-code}, for each $x\in A_{\mathfrak X}$ and every $k\ge 1$ there exist $(b_{-j})_{1\le j \le k}$
and $x_{-k}\in X_{-b_k}$ in such a way that  $x= f_{b_{-1}} \circ f_{b_{-2}} \circ \dots f_{b_{-k}}(x_{-k})$. In consequence, $\underline b=( \dots, b_{-k}, \dots, b_{-2}, b_{-1}) \in \Sigma_{\mathfrak X}^-$ and  $\pi_{\mathfrak X}(\underline b)=x$.

In order to prove item (2), fix $\underline b \in \Sigma^-_{\mathfrak X}$ and 
write $x=\pi_{\mathfrak X}(\underline b) \in A_{\mathfrak X}$. As $A_{\mathfrak X}$ is $F_{\mathfrak X}$-invariant, if $j \in \Psi(x)$ then 
$x\in X_j$ and, by continuity of the map $f_j$,
\begin{align*}
f_j(x)
& = f_j(\lim_{k\to\infty} V_{[\underline b]_k} )   
 = f_j(\lim_{k\to\infty} f_{b_{-1}} \circ f_{b_{-2}} \circ \dots \circ f_{b_{-k}}(X_{b_{-k}}) ) \\
& =  \lim_{k\to\infty} f_j \circ f_{b_{-1}} \circ f_{b_{-2}} \circ \dots \circ f_{b_{-k}}(X_{b_{-k}})
 = \lim_{k\to\infty} V_{[\underline b\ast j]_k} = \pi_{\mathfrak X}(\underline b \ast j).
\end{align*}
  
Since the limit is well defined and $f_j(x)\in \bigcup_{1\le i \le n} f_i(A_{\mathfrak X})=A_{\mathfrak X}$ this proves that $\underline b \ast j \in \Sigma^-_{\mathfrak X}$ and that $\pi_{\mathfrak X}(\underline b \ast j)=f_{j}(x)$, as desired.

\smallskip
Let us now prove item (3). We claim that $\pi_{\mathfrak X}(\Sigma^-_{\mathfrak X})$ is $F_{\mathfrak X}$-invariant. Given $y \in F_{\mathfrak X}(\pi_{\mathfrak X}(\Sigma^-_{\mathfrak X}))$  there exists $x \in \pi_{\mathfrak X}(\Sigma^-_{\mathfrak X})$ and $j\in \{1,\ldots,n\}$  so that $y= f_{j}(x)$ for some $x\in X_j$. Thus, by item (2) one can write $y =\pi_{\mathfrak X}(\underline b \ast j),$ with $x = \pi_{\mathfrak X}(\underline b)$ for some $\underline b \in \Sigma^-_{\mathfrak X}$. Hence, $y \in \pi_{\mathfrak X}(\Sigma^-_{\mathfrak X})$, proving that 
$F_{\mathfrak X}(\pi_{\mathfrak X}(\Sigma^-_{\mathfrak X})) \subset \pi_{\mathfrak X}(\Sigma^-_{\mathfrak X})$. Conversely, if $y\in \pi_{\mathfrak X}(\Sigma^-_{\mathfrak X})$ then there exists $\underline b = (\ldots, b_{-3}, b_{-2}, b_{-1} ) \in \Sigma^-_{\mathfrak X}$ so that 
\begin{align*}
y=\pi_{\mathfrak X}(\underline b) 
    & =
\lim_{k\to\infty} f_{b_{-1}} \circ f_{b_{-2}} \circ \dots f_{b_{-k}}(X_{b_{-k}})    \\
    & = f_{b_{-1}}( \lim_{k\to\infty}  f_{b_{-2}} \circ \dots f_{b_{-k}}(X_{b_{-k}}) ) \\
    & = f_{b_{-1}}( \pi_{\mathfrak X}((\ldots, b_{-4}, b_{-3}, b_{-2} )))
\end{align*}
(here we used the continuity of the map $f_{b_{-1}}$). This proves the claim. Now, as $A_{\mathfrak X} \subseteq \pi_{\mathfrak X}(\Sigma^-_{\mathfrak X})$ (by item (1)), and $A_{\mathfrak X}$ is maximal with this property (recall Theorem~\ref{thm:topological}) then $A_{\mathfrak X} = \pi_{\mathfrak X}(\Sigma^-_{\mathfrak X})$. 

Finally, item (4) is a consequence of items (2) and (3) together with the continuity of the maps $f_j$, with $1\le j \le n$.
This completes the proof of the lemma.
\end{proof}

\subsection{Proof of Theorem~\ref{thm:classification}}
The construction and properties of $\Sigma^-_{\mathfrak X}$ and $\pi_{\mathfrak X}$ have been described above. Therefore, it remains to prove item (2) in the theorem, namely that the contractive local IFS $R_{\mathfrak X}$ is semiconjugate to a local IFS on a shift space given by constraints on the domains of a usual IFS.

Define, for each $1\le j\le n$, the set
\begin{equation}
    \label{eq:defsymbolj}
    B_j:=\{ \underline b \in \Sigma^-_{\mathfrak X} : j \in \Psi(\pi_{\mathfrak X}(\underline b))\} \subseteq \Sigma^-_{\mathfrak X}
\end{equation}
and consider the local IFS $\mathcal S_{\mathfrak B}=(B_j, \tau_{j})_{1\le j \le n}$ given by
$\tau_j(\underline b)= \underline b \ast j$,
for every $\underline b \in B_j$. 

Item (2) in Proposition~\ref{prop: properties code map} guarantees that 
$$
f_{j}(\pi_{\mathfrak X}(\underline b)) = \pi_{\mathfrak X}(\underline b \ast j)
= \pi_{\mathfrak X}(\tau_j(\underline b))
$$
for every $\underline b\in \Sigma^-_{\mathfrak X}$ subject to the constraint 
$ \pi_{\mathfrak X}(\underline b) \in X_{j}$.   
It remains to study the regularity of the surjection $\pi_{\mathfrak X}$. 
By definition, for each $\underline b \in \Sigma^-_{\mathfrak X}$,
$$
\pi_{\mathfrak X}(\underline b)= \lim_{k\to\infty} f_{b_{-1}} \circ f_{b_{-2}} \circ \dots \circ f_{b_{-k}}(X_{b_{-k}})  
$$
In particular, given $\underline a, \underline b\in \Sigma^-_{\mathfrak X}$ one can write $d_-(\underline a,\underline b)=e^{-N}$ for some $N=N(\underline a, \underline b)\in \mathbb N\cup\{\infty\}$. This means that  $a_j=b_j$ for every $-N\le j \le -1$. In consequence, 
$$
d(\pi_{\mathfrak X}(\underline a),\pi_{\mathfrak X}(\underline b)) \le \diam \Big( f_{b_{-1}} \circ f_{b_{-2}} \circ \dots \circ f_{b_{-k}}(X_{b_{-k}})  \Big)
\le \lambda^N \diam(X) = C d_-(\underline a, \underline b)^{\alpha}
$$
where $C=\diam(X)$ and $\alpha= - \log\lambda>0$. This proves that 
$\pi_{\mathfrak X}$ is $\alpha$-H\"older continuous, and finishes the proof of the theorem. 
\hfill $\square$

\section{Dynamics of contractive local IFSs}
\label{sec:orbits}
In this section we will describe the space of orbits for contractive local IFSs, which together with the code space, describes the dynamical behavior of the local attractor. This will prove Theorem~\ref{thmC}.

\subsection{Orbits of a contractive local IFS}
In order to ease the representation of the orbits of a point $x\in X$  we introduce the notation for compositions directed by a symbolic sequence as
$$
f_{\underline a, k}(x):= (f_{a_{k-1}}\circ \cdots \circ f_{a_{0}})(x),
$$
where $a=(a_n)_{n\ge 0} \in \Sigma$.
Moreover, even for $y=f_{\underline a, k}(x) \in \bigcup_{i=1,\ldots,n} X_i$ the future iterates $f_{\underline a, k+1}(x) = f_{j}(y)$ may only be defined for a certain subset of the collection of symbols $j$.
This shows the necessity to introduce  a more general structure which comprises finite orbits and choices of maps depending on the previous points. Let us be more precise.
The iterative process originated by a local IFS $R_{\mathfrak X}=(X_j, f_{j})_{1\le j \le n}$ is determined by:
\begin{enumerate}
\item continuous maps $f_{j}:X_j  \to X$, where $X_j\subset X$; 
\item a map $\Psi: X \to 2^{\{1,\ldots,n\}}$ that associates to each $x\in X$ all available choices of future iteration, that is, 
$\Psi(x)=\{1\le j \le n \colon x\in X_j\}$. 
   \end{enumerate}
In particular, for each sequence $\underline a \in \Sigma$ there are two possibilities for the admissible iterations:
\begin{enumerate}
  \item[(i)] ${\mathcal O}_{\underline a}(x_0)=(x_0, x_1, x_2, \dots, x_n)$ where $x_i=f_{\underline a,i}(x_0)$ for every $1\le i \le n$ but one has 
  $a_{n+1}\notin \Psi(x_n)$ (in this case we say that $x_n$ is an \emph{end point}\index{End point} for the $\underline a$-orbit of $x_0$);
  \smallskip
  \item[(ii)] ${\mathcal O}_{\underline a}(x_0)=(x_0, x_1, x_2, \dots, x_n,\dots)$ where $x_i=f_{\underline a,i}(x_0)$ for every $i\ge 1$.
\end{enumerate}

\medskip
The next proposition shows that all points in the attractor $A_{\mathfrak X}$ are characterized as limits of orbits with arbitrary large length. This will  allow one to derive the conclusion that finite cylinders that fail to intersect the attractor will eventually vanish under the iteration of the local operator. More precisely: 

\begin{proposition}\label{prop: infinite orbits attractor}
   A point $x$ belongs to the local attractor $A_{\mathfrak X}$ if and only if it is accumulated by to orbits with arbitrary large length. Moreover, if a sequence $\underline a\in \Sigma$ verifies $V_{[\underline a]_k}  \cap A_{\mathfrak X}=\emptyset$ then every point in $V_{[\underline a]_k}$ has a finite orbit and is an endpoint.
\end{proposition}
\begin{proof}
Item (3) in Proposition~\ref{prop: properties code map} ensures that 
$\pi_{\mathfrak X}(\Sigma_{\mathfrak X}^-)= A_{\mathfrak X}$. Hence
if $x \in A_{\mathfrak X}$ then there exists $\underline b \in \Sigma_{\mathfrak X}^-$ so that 
$x =
\lim_{k\to\infty} f_{b_{-1}} \circ f_{b_{-2}} \circ \dots f_{b_{-k}}(X_{b_{-k}}).$ In particular, for each $k\ge 1$ there exists $y_k\in X_{b_{-k}}$ and an admissible iteration $(y_k, f_{b_{-k}}(y_k), \ldots, f_{b_{-1}} \circ f_{b_{-2}} \circ \dots f_{b_{-k}}(y_k))$ so that $x_k:=f_{b_{-1}} \circ f_{b_{-2}} \circ \dots f_{b_{-k}}(y_k)$ tends to $x$ as $k$ tends to infinity.
Conversely, assume that $x=\lim_{k\to\infty} x_k$ where, for each $k\ge 1$, there exists a block $(a_0^{(k)}, a_1^{(k)}, \dots, a_k^{(k)})$ such that  
$ 
x_k=f_{a_{k}^{(k)}} \circ \dots f_{a_{0}^{(k)}}(y_k)
    \in F_{\mathfrak X}^k(X).
$ 
As $A_{\mathfrak X}=\lim_{k\to\infty} F_{\mathfrak X}^k(X)$ (recall Theorem~\ref{thm:topological}) it follows immediately that $x\in A_{\mathfrak X}$.
The second statement is an immediate consequence of the latter. 
\end{proof}

The next proposition shows that there exist at least a point in the attractor with an infinite orbit, hence it belongs to $A^\infty_{\mathfrak X}$ (recall there may exist points in the attractor where such property fails, cf. Example~\ref{ex:shift}).

\begin{proposition}
	\label{prop:keyorbit}
	 Consider a local  IFS $R_{\mathfrak X}=(X_j,f_j)_{1\le j \le n}$. The following properties are equivalent:
     \begin{enumerate}
         \item the set $A_{\mathfrak X}$ is non-empty;
         \item there exists a point $x\in X$ 
         with an infinite orbit.
     \end{enumerate}
     In particular, if any of the conditions hold then  $A^\infty_{\mathfrak X} \neq \emptyset$. 
\end{proposition}
\begin{proof}
First we prove the equivalence.

\medskip
\noindent $(1)\Rightarrow (2)$
\smallskip

Define, for each $x \in  A_{\mathfrak X}$,
	\[\Gamma(x)= \bigcup_{k\ge 1} \{ ((x_0,a_0), (x_1,a_1), \dots , (x_{k-1}, a_{k-1})) | x_0=x, \; a_{j}\in \Psi(x_j) , \;  0 \leq j \leq k-1\},\]
	the set of finite orbits starting at $ x_0=x$. Set also $\Gamma=\bigcup_{x\in  A_{\mathfrak X}} \Gamma(x)$.
	
	We claim that there must exist points $x \in  A_{\mathfrak X}$ such that $\Gamma(x)$ contains orbits with arbitrarily large length.  Indeed, as $A_{\mathfrak X} \neq \emptyset$ one can pick a point $z \in A_{\mathfrak X} = \bigcup_{1\le j \le n} f_{j}(A_{\mathfrak X} \cap X_{j})$ so that there exists $x_0 \in A_{\mathfrak X} $ and $b_0 \in \{1,\ldots,n\}$ with $f_{b_0}(x_0)= z$. Repeating this procedure, there exists $x_1 \in A_{\mathfrak X} $ and $b_1 \in \{1,\ldots,n\}$ with $f_{b_1}(x_1)= x_0$.
	Recursively, for each $m\ge 1$ one can find,  $1\le b_0,b_1, \dots, b_{m-1} \le n$ with
	\[f_{b_0}(f_{b_1}(\cdots (f_{ b_{m-1}}(x_{m-1}))))=z, \text{ where } x_{m-1} \in  A_{\mathfrak X}\]
	that is, $( (x_{m-1}, b_{m-1}), \dots,  (x_{1},b_{1}),  (x_0,b_0))  \in \Gamma(x_{m-1})$.
	
	Using this property one can choose a sequence $(w_{i})_{i\ge 1}$ in $\Gamma$ where each $w_i$ is a finite orbit and whose lengths for an increasing sequence $k_1 < k_2< k_3 <\ldots$. One writes
	\[w_1:=((x_0^{1},a_0^{1}), (x_{1}^{1},a_{1}^{1}), \dots , (x_{k_1}^{1}, a_{k_1}^{1}))  \in \Gamma(x_0^{1});\]
	\[w_2:=((x_0^{2},a_0^{2}), (x_{1}^{2},a_{1}^{2}), \dots , (x_{k_2}^{2}, a_{k_2}^{2}))  \in \Gamma(x_0^{2});\]
	and so on. This will be used to construct a sequence $\underline a :=(a_{0}^*, a_{1}^*, \ldots) \in \Sigma$. This is done by induction. For $j=0$ we notice that $(a_0^{i})$ is a sequence taking finitely many values, thus by the pigeonhole principle at least one of them must repeat infinitely many times. Take $a_{0}^*$ as one of such indices. Up to consider a subsequence of $(w_{i})_{i\ge 1}$ one can assume that $a_0^i=a_0^*$ for every $i\ge 1$. Following this diagonal process we will find by induction a subsequence of $(w_{i})_{i\ge 1}$ (we keep denoting it by $(w_{i})_{i\ge 1}$ for notational simplicity) so that 
	where $a_j^{i} = a_{j}^*$ for every  $i$.
	
	Now, as the local attractor  $ A_{\mathfrak X}$ is compact one can  find a point  $x_* \in  A_{\mathfrak X}$ 
    as accumulation point of $(x_0^{i})_{i\ge 1}$.  Moreover, using that every function  $f_{a_{j}^*}$ is continuous and each $X_{a_{j}^*}$  is closed, we obtain $a_{j}^*\in \Psi(x_j^*)$ where $x_{j+1}^*= f_{a_{j}^*}(x_{j}^*), \; j \geq 0$. In conclusion
	\[((x_0^*,a_0^*), (x_1^*,a_1^*) , (x_{2}^*, a_{2}^*), \dots) \in \Omega\]
    and $x_0^* \in A^\infty_{\mathfrak X} \subset A_{\mathfrak X}$.

\medskip
\noindent $(2)\Rightarrow (1)$
\smallskip

Suppose that $x \in X$ is such that $x_0=x$ and $x_{i}=f_{a_{i}}(x_{i-1})$ for $i \geq 1$. We notice that $F_{\mathfrak X}^0(X)=X$ and
    \[x_{i}=f_{a_{i}}(f_{a_{i-1}}(\cdots (f_{a_{1}}(x)))) \in F_{\mathfrak X}^i(X), \quad i \geq 1.\]
As $F_{\mathfrak X}^k(X)$ is always a non-empty compact set for each $k \geq 0$ then  
$A_{\mathfrak X} = \bigcap_{k \geq 0} F_{\mathfrak X}^k(X) \neq\emptyset$ 
by the Cantor intersection theorem. 

\medskip
This prove the equivalence between items (1) and (2).  The second statement in the proposition follows from the proof $(1)\Rightarrow (2)$, as the existence of a point with an infinite orbit led to the construction of a point in $A^\infty_{\mathfrak X} $.
	This completes the proof of the proposition.	  	  
\end{proof}

\subsection{Proof of Theorem~\ref{thmC}}
The first claim in the theorem follows from Proposition~\ref{prop:keyorbit}. 
By construction of $\Omega$ all points in the range of $\pi_{0,X}$ 
can be iterated indefinitely by the contractive local IFS, and 
$\displaystyle\lim_{i \to\infty} \text{dist}_H(x_i,A_{\mathfrak X})=0$
for each $((x_i,a_i))_{i\ge 0}\in \Omega$.
Now, consider the enlarged shift map
$\bar \sigma: \Omega \to \Omega$ defined by ~\eqref{eq:extendedsigma}
and the usual shift $\sigma:\Sigma^+ \to \Sigma^+$.
It is simple to check that the diagram
\begin{equation}
    \label{eq:dyn-diagram}
\begin{array}{rcl}
    \Omega & \xrightarrow[]{\bar \sigma} & \Omega\\
    \pi_{\Sigma^+} \downarrow   &  & \downarrow \pi_{\Sigma^+} \\
     \Sigma^+ & \xrightarrow[\sigma]{} & \Sigma^+
\end{array}
\end{equation}
commutes (that is, $\pi\circ \bar\sigma = \sigma\circ \pi$),
where the projection
$\pi_{k,\Sigma^+}(((x_i,a_i))_{i\ge 0})=a_k$
and 
$$
\pi_{\Sigma^+}=(\pi_{0,\Sigma^+},\pi_{1,\Sigma^+}, \pi_{2,\Sigma^+}, \dots).
$$ 
Let us use the following auxiliary result:
\begin{lemma}
    \label{le:inv-attractor}
    Let $R_{\mathfrak X}=(X_j, f_{j})_{1\le j \le n}$ be a constractive local IFS and assume that $A_{\mathfrak X}\neq \emptyset$.
    The following properties hold:
    \begin{enumerate}
        \item $  
    \Omega_\infty :=  
    \bigcap_{k\ge 0} \bar\sigma^k(\Omega)
    $
    is a $\bar \sigma$-invariant subset of $(X\times \{1,2, \dots, n\})^{\mathbb N}$;
    \smallskip
    \item $\Sigma_\infty:=\pi_{\Sigma^+}(\Omega_\infty)$ is a $\sigma$-invariant subset of $\Sigma^+$.
    \end{enumerate} 
     Moreover, if $\Omega$ is compact then both sets $\Omega_\infty$ and $\Sigma_\infty$ are compact.
\end{lemma}

\begin{proof}
By construction $\bar \sigma(\Omega)\subset \Omega$. In particular the sequence $(\bar \sigma^k(\Omega))_{k\ge 0}$ 
is nested and, by continuity of $\bar \sigma$, 
$$
\bar \sigma \Big(\bigcap_{k\ge 0} \bar\sigma^k(\Omega)\Big)
=\bar \sigma \Big(\lim_{k\to\infty} \bar\sigma^k(\Omega)\Big)
=\lim_{k\to\infty} \bar\sigma^{k+1}(\Omega)
=\bigcap_{k\ge 0} \bar\sigma^k(\Omega).
$$
This proves that $\Omega_\infty:=\bigcap_{k\ge 0} \bar\sigma^k(\Omega)$ is a $\bar\sigma$-invariant subset. 

The $\sigma$-invariance of the set $\Sigma_\infty$ follows from the commutative diagram ~\eqref{eq:dyn-diagram}.
Finally, if $\Omega$ is a compact set then $\Omega_\infty$ is a  nested and non-empty intersection of compact sets (cf. Proposition~\ref{prop:keyorbit}),   hence compact, while $\Sigma_\infty$ is the image of such compact set under the continuous map $\pi_{\Sigma^+}$. This finishes the proof of the lemma.   
\end{proof}

Given $x\in X$ consider the space 
$$
\Omega^x_\infty = \pi_{0,X}^{-1}(\{x\})= \Big\{ ((x_i,a_i))_{i\ge 0} \in \Omega_\infty \colon \pi_{0,X}(((x_i,a_i))_{i\ge 0})=x \Big\}
$$
which represents all possible orbits $(f_{a_{k}}\circ \ldots f_{a_{1}} \circ f_{a_{0}}(x))_{k\ge 0}$ of the point $x$, for all possible choices of paths $\underline a \in \Sigma^+$.

We now proceed to define the space of two-sided sequences in $\Sigma^-_{\mathfrak X} \times \Sigma_\infty$ so that the point $x\in X$ determined by the past coordinates (code space) has an admissible future coordinates. More precisely, let us define the set
\begin{equation}
    \label{defOmegax}
\widehat \Sigma_{\mathfrak X}=\Big\{(\underline b, \underline a) \in \Sigma^-_{\mathfrak X} \times \Sigma_\infty \colon \underline a \in \pi_{\Sigma^+}(\Omega^x_\infty) \; \text{where} \; x=\pi_{\mathfrak X}(\underline b) \Big\}.
\end{equation}
By construction $\pi_{\mathfrak X}\circ\pi_{\Sigma^- }(\widehat \Sigma_{\mathfrak X})=A^\infty_{\mathfrak X}$,  
hence $\widehat\Sigma_{\mathfrak X}$ is non-empty as 
$A^\infty_{\mathfrak X}\neq\emptyset$. 
In order to finish the proof of Theorem~\ref{thmC} it remains to prove the following:

\begin{proposition}
    \label{prop:shiftwelldef}
The shift $\widehat \sigma \colon \widehat \Sigma_{\mathfrak X} \to \widehat \Sigma_{\mathfrak X}$ given by
$$
\widehat \sigma ( \dots , b_{-3}, b_{-2}, b_{-1}, \boxed{a_0}, a_1, a_2, a_3,\dots)
    = 
     ( \dots , b_{-2}, b_{-1}, a_0, \boxed{a_1}, a_2, a_3, a_4, \dots)
$$
is well defined. Moreover, for every $(\underline b, \underline a) \in \widehat\Sigma_{\mathfrak X}$,
$$
\widehat \sigma(\underline b, \underline a)=(\underline b \ast a_0, \sigma(\underline a)) = (\tau_{a_0}(\underline b), \sigma(\underline a)).
$$
\end{proposition}

\begin{proof}
    It is enough to show that for each two-sided sequence
    $(\underline b, \underline a)\in \widehat \Sigma_{\mathfrak X}$ the sequence $(\underline b \ast a_0,\sigma (\underline a))$ also belongs to $\widehat \Sigma_{\mathfrak X}$. Indeed, fix an arbitrary $(\underline b, \underline a)\in \widehat \Sigma_{\mathfrak X}$ and write $x=\pi_{\mathfrak X}(\underline b)$. By definition (recall ~\eqref{defOmegax}), as $\underline a\in \pi_{\Sigma^+}(\Omega^x_\infty)$ one has that $a_0\in \Psi(x)$. Moreover, using item (2) in Proposition~\ref{prop: properties code map}, 
$$
f_{a_0}(x)=f_{a_0}(\pi_{\mathfrak X}(\underline b))
= \pi_{\mathfrak X}(\underline b\ast a_0) 
$$
hence $\underline b\ast a_0\in \Sigma^-_{\mathfrak X}$. Furthermore, as $\underline a\in \pi(\Omega^x_\infty)$ then the orbit
$$
x_{i}=f_{a_i}(x_{i-1})
$$
is well defined for every $i\ge 1$, hence it follows immediately that
$\sigma (\underline a)\in \Omega^{f_{a_0}(x)}_\infty. $ As the last claim is immediate, this proves the proposition. 
\end{proof}

\medskip

\section{On a skew dynamical structure for local IFS}\label{sec:skewtype}

This section is devoted to the proof of Theorem~\ref{thm:D}.
Consider a compact metric space $(X,d)$, a finite collection of closed subsets $\mathfrak X =(X_j)_{1\le j \le n}$ of $X$ and a  contractive local IFS $R_{\mathfrak X}=(X_j,f_j)_{1\le j\le n}$.

For formal purposes one can consider a slightly different dynamical system written as $R_{\mathfrak X^{e}}:=(X_j,f_j)_{0\le j\le n}$ with the following collection of  subsets $\mathfrak X^{e} =(X_j^{e})_{0\le j \le n}$ of $X$ where $X_j^{e}:= X_j,\; 1 \le  j\le n$ and $X_0^{e} := X \backslash \bigcup_{1 \le  j\le n} X_j$ and $f_0(x):=x, \forall x \in X_0^{e}$. One denote the local IFS $R_{\mathfrak X^{e}}$ as the natural extension of $R_{\mathfrak X}$.  

\begin{remark}
We note that the local IFS $\mathfrak X^{e}$ is not contractive because $f_0$ is the identity and $X_0^{e}$ may not be a closed set. In particular, the local fractal operator aassociated to $\mathfrak X^{e}$ may not preserve compact sets. This is not a problem because we are not interested in its dynamics. We only use it as a tool to represent the whole set of orbits of the original local IFS in a useful manner.     
\end{remark}

The main property of the natural extension is that $\bigcup_{0 \le  j\le n} X_j^{e} = X$, therefore all positive semi-orbits of $R_{\mathfrak X^{e}}$ are infinite. More precisely, the new address function is given by
\[\Psi^{e}(x)=
\begin{cases}
    \begin{array}{ll}
\Psi(x)          & , \text{ for } x \in  \bigcup_{1 \le  j\le n} X_j \\
        0 & , \text{ otherwise}
    \end{array}.
\end{cases}
\]

In this way, if 
\[((x_0,a_0), (x_1,a_1), \ldots, (x_{k-1},a_{k-1}) )\]
is a finite orbit of $R_{\mathfrak X}$, that is, $\Psi(x_{k})=\emptyset$ or equivalently $\displaystyle x_{k}= f_{a_{k-1}} (x_{k-1}) \in X/ \bigcup_{1 \le  j\le n} X_j = X_0^{e}$ is an endpoint,  then,  $\Psi^{e}(x_{k-1})=0$ and the orbit continues for $R_{\mathfrak X^{e}}$, given by
\[((x_0,a_0), (x_1,a_1), \ldots, (x_{k-1},a_{k-1}) , (x_{k},0), (x_{k},0),  (x_{k},0),  (x_{k},0),  (x_{k},0), \ldots).\]

We now consider  a space big enough to contain all the finite and infinite orbits of  $R_{\mathfrak X}$ and, at the same time, the code space, denoted the extended shift space 
\[\Omega^{e}:= (X \times \{0,1,\ldots,n\})^{\mathbb{Z}}\]
usually decomposed in past and future sequences, respectively, 
\[ \Omega^{e}_{-}:=(X \times \{0,1,\ldots,n\})^{-\mathbb{N}} 
\text{ and }
\Omega^{e}_{+}:=(X \times \{0,1,\ldots,n\})^{\mathbb{N}_0}.\]
The same can be defined for the non-extended version
\[ \Omega_{-}:=(X \times \{1,\ldots,n\})^{-\mathbb{N}} 
\text{ and }
\Omega_{+}:=(X \times \{1,\ldots,n\})^{\mathbb{N}_0}.\]
It is useful to represent this spaces as a direct sum
\[\Omega^{e}= \Omega^{e}_{-} \oplus \Omega^{e}_{+}:=\{(z|w) \colon z \in \Omega^{e}_{-}  \text{ and }  w \in \Omega^{e}_{+}\}\]
and define the natural projections $\pi^{-}: \Omega^{e} \to \Omega^{e}_{-}$ and $\pi^{+}: \Omega^{e} \to \Omega^{e}_{+}$ defined by  $\pi^{-}(z|w) = z$ and  $\pi^{+}(z|w) = w$, for all $(z|w)  \in \Omega^{e} $.

Define 
$\Sigma^{-}_{e}=\{0,1,  \dots, n\}^{-\mathbb N}$ endowed with the metric
$$ 
d_-(\underline a, \underline b) = \exp\Big(\,-\sup\{N\ge 1\colon a_j = b_j, \;\text{for every}\; -N\le j \le -1 \}\Big)
$$ 
for every $\underline a, \underline b\in \Sigma^-$, where $\sup \emptyset =0$. Analogously, we define $ \Sigma^{-}=\{1,2, \dots, n\}^{-\mathbb N}$ with the same metric structure. Then, naturally,  $\Sigma^- \subset \Sigma^{-}_{e}$.

It is also important to define some additional projections
\[\pi_{k,\Sigma^{-}} (z|w):= a_{-k}\in\{0,1,\ldots,n\} , \; k \le -1; \]
\[\pi_{\Sigma^{-}}(z|w):= (\ldots, a_{-2},a_{-1} ) \in \Sigma^{-}_{e};\]
\[\pi_{k,X} (z|w):= x_{k}\in X, \; k \in \mathbb{Z} ;\]
\[\pi_{X} (z|w):=(x_0, x_1, x_{2}  \ldots)\in X^{\mathbb{N}_0},\]
where $(z|w):=((\ldots, (x_{-2},a_{-2}), (x_{-1},a_{-1}) )| ((x_0,a_0), (x_1,a_1),(x_{2},a_{2}) ,   \ldots)) \in \Omega^{e}$.
Note that  $\pi_{k,X} (z|w) = \pi_{0,X}(\sigma^{k} (z|w))$ so $\pi_{0,X}$ is the only necessary projection on $X$.

Finally, we define the double-sided shift map $\sigma: \Omega^{e} \to \Omega^{e}$ by
\[\sigma(z|w):=((\ldots, (x_{-2},a_{-2}), (x_{-1},a_{-1}),(x_0,a_0) )| ( (x_1,a_1),(x_{2},a_{2}) ,   \ldots)) \]
for $(z|w):=((\ldots, (x_{-2},a_{-2}), (x_{-1},a_{-1}) )| (x_0,a_0), (x_1,a_1),(x_{2},a_{2}) ,   \ldots)) \in \Omega^{e}$. The one-sided shift in the future coordinates is defined by the same symbol 
\[\sigma(w):= ( (x_1,a_1),(x_{2},a_{2}) ,   \ldots) \]
and  the concatenated past is denoted by
\[\tau_{w}(z):= (\ldots, (x_{-2},a_{-2}), (x_{-1},a_{-1}),(x_0,a_0) ).\]

\medskip
Those spaces and projections do not take in to account the dynamics of the local IFS $R_{\mathfrak X^{e}}$ (in particular of $R_{\mathfrak X}$). In order to do that, one need to consider a dynamically defined subset of $\Omega^{e}$ whose function is to construct an autonomous dynamics containing the orbits of the IFS together with the sequence of indices defining it. This set is denoted the extended space of positive semi-orbits $\Omega^{e}_{+} (\mathfrak X) \subseteq \Omega^{e}_{+}$ given by
\begin{equation}\label{eq:defOmegae}
    \Omega^{e}_{+} (\mathfrak X) :=\{ ((x_0,a_0), (x_1,a_1),(x_{2},a_{2}) ,   \ldots)  \in \Omega^{e}_{+}\;|\; a_{i} \in  \Psi^{e}(x_{i}) , i \ge 0\}.
\end{equation}
We notice  that, the potentially smaller, set of all infinite orbits of the original local IFS, $\Omega_{+}(\mathfrak X) :=\{ ((x_0,a_0), (x_1,a_1),(x_{2},a_{2}) ,   \ldots) \;|\; a_{i} \in  \Psi(x_{i}) , i \ge 0\}$ is always not empty in the contractive setting (see Proposition~\ref{prop:keyorbit}), but regrettably it does not contains endpoints thus it may not suitable to represent the whole local attractor $A_{\mathfrak X}$.

We denote by $\ell: \Omega^{e}_{+} (\mathfrak X) \to \mathbb{N}_0 \cup \{\infty\}$ the length of an element \[w=((x_0,a_0), (x_1,a_1),(x_{2},a_{2}) ,   \ldots) \in \Omega^{e}_{+} (\mathfrak X)\] as the minimum index $k$ such that $a_{i}=0, \forall i \ge k$.  

We notice that $\ell(w)=0$ means that \[w= ((x_0,0), (x_0,0),(x_{0},0) ,   \ldots).\]
On the other hand  $1\le \ell(w)=k< \infty$ it means that \[w=( (x_0,a_0), \ldots, (x_{k-1},a_{k-1}), (x_{k},0) ,  (x_{k},0) ,  \ldots)\] with $a_0, \ldots, a_{k-1} \neq 0$. Finally, if $\ell(w)=\infty$ then we obtain $ a_{k} \neq 0, \forall k \ge 0$, that is, $w \in \Omega_{+} (\mathfrak X)$.

From the start,  we do not require that $a_{i} \in  \Psi^{e}(x_{i}) , i \le -1$, but after $k$ iterations by $\sigma$ we get $a_{i} \in  \Psi^{e}(x_{i}) , -k \leq i \le -1$.  

Thus,  the space of interest is 
\[\Omega^{e}_{-}   \oplus \Omega^{e}_{+} (\mathfrak X)   =  (\pi^{+})^{-1}(\Omega^{e}_{+} (\mathfrak X)) \subseteq \Omega^{e} .\]

\begin{definition}
	We say that $(z|w) \in \Omega^{e}_{-}   \oplus \Omega^{e}_{+} (\mathfrak X)$ is  $k$-admissible, for $k\ge 0$, if 
	$\pi^{+}(\sigma^{-k}((z|w))) \in \Omega^{e}_{+} (\mathfrak X) ,$ and  $\pi_{1, \Sigma^{-}_{e}}(z|w) \neq 0$.
\end{definition}
We notice that  all elements of $\Omega^{e}_{-}   \oplus \Omega^{e}_{+} (\mathfrak X)$ different from $((\ldots,(x_{0}, 0)) \,|\, w)$ are $0$-admissible.

Observe that $\pi_{\Sigma^{-}_{e}}(z|w):= (\ldots, a_{-2},a_{-1} ) \in \Sigma^{-}$ if, and only if, $a_{-i} \neq 0, \forall i$.

However, the dynamical system generated by the local IFS will be better represented by the set
$ 
    \Omega^{e} (\mathfrak X) \subset \Omega^{e}_{-}   \oplus \Omega^{e}_{+} (\mathfrak X)
$ 
given by
\begin{equation}
    \label{eq:defOmegaeu}
\Omega^{e} (\mathfrak X):= \bigcap_{k \geq 1} \{(z|w) \colon  \text{is $k$-admissible}\}.
\end{equation}

\subsection{Proof of Theorem~\ref{thm:D}}
\begin{proof}
We prove the items separately.

\medskip
(1) We claim that
$\pi_{\Sigma^{-}} \Big(\Omega^{e} (\mathfrak X)\Big) \subset \Sigma^{-}_{\mathfrak X}$.
By construction of $\Omega^{e}(\mathfrak X)$, one has that $\pi_{\Sigma^{-}_{e}}(z|w)= (\ldots, b_{-2},b_{-1} ) \in \Sigma^{-}$, hence
    $\pi_{\Sigma^{-}_{e}}(z|w)=\pi_{\Sigma^{-}}(z|w)$ for any  $(z|w) \in \Omega^{e} (\mathfrak X)$.   
Moreover, if $(z|w) \in \Omega^{e} (\mathfrak X)$, then for any $k \geq 1$ one has 
	\[ \pi^{+}(\sigma^{-k}((z|w))) = (  (x_{-k},b_{-k}), \ldots,  (x_{-1},b_{-1}), (x_0,a_0), \ldots, (x_{i-1},a_{i-1}), (x_{i},0),  \ldots) \in  \Omega^{e}_{+} (\mathfrak X) \]
	meaning that  for $\underline b := \pi_{\Sigma^{-}}((z|w))$ one has  $x_0 \in V_{[\underline b]_k} =f_{b_{-1}} \circ f_{b_{-2}} \circ \dots f_{b_{-k}}(X_{b_{-k}})  \neq \emptyset$. Hence  $\underline b \in \Sigma^{-}_{\mathfrak X} $. 

Conversely, in order to prove $\Sigma^{-}_{\mathfrak X}\subset \pi_{\Sigma^{-}} \Big(\Omega^{e} (\mathfrak X)\Big)$,
given $\underline b \in \Sigma^{-}_{\mathfrak X}$ and $x_0=\pi_{\mathfrak X}(b)\in A_{\mathfrak X}$
one can consider the positive orbit from $x_0$, if we  have an endpoint then  the positive orbit of $x_0$  is either
	\[w:=( (x_0,a_0), \ldots, (x_{i-1},a_{i-1}), (x_{i},0) ,  (x_{i},0) ,  \ldots) \in  \Omega^{e}_{+} (\mathfrak X) \] 
	with $a_0, \ldots, a_{i-1} \neq 0$
    or an infinite positive orbit $w$ of $x_0$ without endpoints ($\Psi^{e}(x_{i})\neq 0, i \ge 0$). In any case
    $(z|w) \in \Omega^{e} (\mathfrak X)$ and 
$\pi_{\Sigma^-}((z|w))=\underline b$.

\medskip
(2) By assumption $A_{\mathfrak X}\neq \emptyset$. 
We claim that $\pi_{0,X} (\Omega^{e}(\mathfrak X)) \supseteq  A_{\mathfrak X}$.
As above, given any $x_0 \in A_{\mathfrak X}$ one can find an infinite sequence $z=(\ldots, (x_{-k},b_{-k}), \ldots,  (x_{-1},b_{-1}))$  satisfying $b_{-k} \in \Psi(x_{-k}) , \; k \ge 1$ and an infinite positive orbit $w\in \Omega^{e}_{+} (\mathfrak X)$.
	Clearly, $\pi^+(\sigma^{-k}((z|w))) \in \Omega^{e}_{+} (\mathfrak X) $ for all $k \geq 0$, that is,  $(z|w) \in \Omega^{e} (\mathfrak X)$. By construction
    $\pi_{0,X}((z|w))=x_0$, proving the claim.
The inclusion
$\pi_{0,X} (\Omega^{e}(\mathfrak X)) \subseteq  A_{\mathfrak X}$ is a direct consequence of item (1).

\medskip
(3) We proceed to show that the diagram
\begin{center}
			\begin{tabular}{ccc}
				$\Omega^{e}(\mathfrak X) $ & $\stackrel{\sigma^{-1}}{ \longrightarrow }$ & $\Omega^{e}(\mathfrak X)$ \\
				$  \pi_{-1,X} \times \pi_{\Sigma^{-}}  \downarrow$ &   & $\downarrow \pi_{1,X} \times \pi_{\Sigma^{-}}$ \\
				$ X \times \Sigma^{-} $             & $\stackrel{S}{\longrightarrow } $ &    $X \times \Sigma^{-} $ \\
			\end{tabular}
		\end{center}
        commutes,
		where $S(x, \underline b) = ( f_{b_{-1}}(x), \sigma^{-}(\underline b))$. 
Take an arbitrary \[(z|w):=(  \ldots, (x_{-2},b_{-2}), (x_{-1},b_{-1})|(x_0,a_0), (x_{1},a_{1}),   \ldots) \in \Omega^{e} (\mathfrak X).\]  
Then
	\[ S((\pi_{-1,X} \times \pi_{\Sigma^{-}})(z|w))= S(x_{-1}, \underline b) =( f_{b_{-1}}(x_{-1}), \sigma^{-}(\underline b) ) = (x_0, (\ldots, b_{-3}, b_{-2})) \]
	and
\begin{align*}
   (\pi_{1,X} \times \pi_{\Sigma^{-}} )(\sigma^{-1}(z|w))
   & =(\pi_{1,X} \times \pi_{\Sigma^{-}} ) (  \ldots, (x_{-2},b_{-2})|(x_{-1},b_{-1}),(x_0,a_0), (x_{1},a_{1}),   \ldots)  \\
   & = (x_0, (\ldots, b_{-3}, b_{-2})),
\end{align*}
	proving the commutativity.
\end{proof}

\section{Exponential convergence on the basin of the local attractor}
\label{sec:exponentialbasin}

This section is devoted to the proof of Theorem~\ref{thmB}, concerning the local attractor of a contractive local IFS. The first step is the following
(we refer the reader to \cite[Proposition 2.4]{BHM14} for a slightly more restrictive setting):

\begin{lemma}\label{le:global and local attractors}
   Let $R= \left(X, f_j \right)_{1\le j \le n}$ be a global IFS with attractor $A_{R}$. Assume that $X_j\subset X$ for each $1\le j \le n$ and that $R_{\mathfrak X}=(X_j, f_{j})_{1\le j \le n}$ is the corresponding local IFS. Then, $A_{\mathfrak X} \subseteq A_{R}$.
\end{lemma}
\begin{proof}
The proof is by induction. Take $A_0=X$ and $B_0=X$ and note that $A_0 =B_0$.
Suppose that $A_{k} \subseteq B_{k}$, for $F (B_{k})= B_{k+1}$ and $A_{k+1}= F_{\mathfrak X}(A_{k})$ for some $k\ge 1$.
Using that, by construction $F_{\mathfrak X}(B) \subseteq F(B)$ for any $B\subset X$,
 then
   $$A_{k+1}= F_{\mathfrak X}(A_{k}) \leq F_{\mathfrak X}(B_{k}) \leq F (B_{k})= B_{k+1}.$$
This proves that $A_n \subset B_n$ for every $n\ge 1$ and, consequently,
$$
A_{\mathfrak X} = \lim_{n\to\infty} A_n \subseteq
\lim_{n\to\infty} B_n =A_R.
$$
\end{proof}

The next corollary shows that, for a contractive IFS in a complete metric space, we could restrict our analysis to a compact metric space, the global attractor, and to obtain the local attractor via iteration in that space.  This justifies the assumption of always assuming that the underlying space $X$ is a compact metric space.

\begin{corollary}\label{cor: iterative attractor full attractor}
   Assume $A_{\mathfrak X} \neq \emptyset$. Then,
   $$A_{\mathfrak X}=\bigcap_{k \geq 0} F_{\mathfrak X}^{k}(A_{R}) = \lim_{k \to \infty}  F_{\mathfrak X}^{k}(A_{R}),$$
   that is, $A_{R} \in \mathcal{B}_{\mathfrak X}^{\, \text{inv}}$. 
    
\end{corollary}
\begin{proof}
   From Lemma~\ref{le:global and local attractors} one obtains the inclusion 
   $A_{\mathfrak X} \subseteq A_{R}$.

   On the other hand, $A_0=A_{R}$ is $F_{\mathfrak X}$-positively invariant  because
    $A_1=F_{\mathfrak X}(A_0) \subseteq F(A_0)= F(A_{R})=A_{R}=A_0.$ 

   Thus, by Proposition~\ref{prop:set prop F_loc} one concludes that  $A_{\mathfrak X}=\bigcap_{k \geq 0} A_{k} = \lim_{k \to \infty} A_{k}$.
\end{proof}

In the next result one establishes that every positively invariant set that is attracted to a local attractor is indeed attracted with exponential velocity. In particular, this is extremely useful from the  computational viewpoint in order to approximate the local attractor. 

\medskip

	\begin{proposition}\label{prop: contrac of sequ loc attrac beta version}
	Let $R_{\mathfrak X}=(X_j, f_{j})_{1\le j \le n}$ be local IFS with local attractor $A_{\mathfrak X}$. There exists $\lambda\in (0,1)$ so that, for any  $A_0 \in \mathcal{B}_{\mathfrak X}$ and any $k\ge 1$,
	$$\text{dist}_H(\rm{ess}(F_{\mathfrak X}^k(A_0)), A_{\mathfrak X}) \leq \lambda^k \text{diam}(X).$$
\end{proposition}
\begin{proof}
	 Since $A_{\mathfrak X} \subset A_0$ we obtain 
     $V_{[\underline b]_k}(A_{\mathfrak X})\subset V_{[\underline b]_k}(A_0)$ for every $\underline b\in \Sigma^-$ and $k\ge 1$, where we write
     $ 
     V_{[\underline b]_k}(K)= f_{b_{-1}} \circ f_{b_{-2}} \circ \dots f_{b_{-k}}(K)
     $ 
     for each compact set $K\subseteq X$.
     Consequently, $A_{\mathfrak X} = F_{\mathfrak X}^i(A_{\mathfrak X})  \subseteq  F_{\mathfrak X}^i(A_{0})$. We claim that 
     $$
     \text{dist}_H(V_{[\underline b]_k}(A_{\mathfrak X}), V_{[\underline b]_k}(A_0)) \leq \lambda^k \text{diam}(X)
     $$ 
     for every $\underline b \in \Sigma_{\mathfrak X}^-$ and every $k\ge 1$.
     By the inclusion $V_{[\underline b]_k}(A_{\mathfrak X})\subset V_{[\underline b]_k}(A_0)$ it is enough to show that for every $x\in  V_{[\underline b]_k}(A_0)$ there exists
     $z\in V_{[\underline b]_k}(A_{\mathfrak X})$
so that $\text{dist}(x,z)\le \lambda^k \text{diam}(X).$
In fact, since 
$V_{[\underline b]_k}(A_{\mathfrak X})\neq \emptyset$
one can choose $z\in V_{[\underline b]_k}(A_{\mathfrak X})$. As $F_{\mathfrak X}(A_{\mathfrak X})=A_{\mathfrak X}$ one knows that $z\in A_{\mathfrak X}$ and 
$z=f_{b_{-1}}(\cdots(f_{b_{-k}}(z')))$
for some $z'\in A_{\mathfrak X}$.
Similarly, $x=f_{b_{-1}}(\cdots(f_{b_{-k}}(x')))$ for some $x'\in A_0$. Hence,
using the common contraction rate $\lambda$ we obtain $d(x,z) < \lambda^k d(x',z') \leq  \lambda^k \text{diam}(X)$.
This shows that $V_{[\underline b]_k}(A_0)$ is contained in the $\lambda^k \text{diam}(X)$ neighborhood of the set $V_{[\underline b]_k}(A_{\mathfrak X})$.
Now, observe that $F_{\mathfrak X}^k(A_{0})$ can be written as a union of such cylinder sets as 
\begin{align}
    F_{\mathfrak X}^k(A_{0}) 
    & = \bigcup_{V_{[\underline b]_k}(A_{\mathfrak X}) =\emptyset} V_{[\underline b]_k}(A_0) 
    \cup  \bigcup_{V_{[\underline b]_k}(A_{\mathfrak X})  \neq\emptyset} V_{[\underline b]_k}(A_0) \nonumber \\
    & = \bigcup_{V_{[\underline b]_k}(A_{\mathfrak X}) =\emptyset} V_{j_{-i},\ldots,j_{-1}}(A_{0}) \cup  \rm{ess}(F_{\mathfrak X}^i(A_{0}))
    \label{eq:extraterm}
\end{align}
while
\begin{align*}
    A_{\mathfrak X} = F_{\mathfrak X}^i(A_{\mathfrak X}) 
    & = \bigcup_{V_{[\underline b]_k}(A_{\mathfrak X}) =\emptyset} V_{[\underline b]_k}(A_{\mathfrak X}) \cup  \bigcup_{V_{[\underline b]_k}(A_{\mathfrak X})  \neq\emptyset} V_{[\underline b]_k}(A_{\mathfrak X}) \\
	  & =\bigcup_{V_{[\underline b]_k}(A_{\mathfrak X}))  \neq\emptyset} V_{[\underline b]_k}(A_{\mathfrak X}).
\end{align*}
Moreover, as $\text{dist}_H(\bigcup_{m=1}^n A_m,\bigcup_{m=1}^n B_m) \le \max_{1\le m \le n} \text{dist}_H(A_m,B_m)$ (see e.g. \cite{Fal03}) then
      \begin{align*}
        \text{dist}_H(A_{\mathfrak X} ,  \rm{ess}(F_{\mathfrak X}^k(A_{0}))) 
        & = \text{dist}_H\Bigg(\bigcup_{V_{[\underline b]_k}(A_{\mathfrak X}) \neq\emptyset} V_{[\underline b]_k}(A_{\mathfrak X}), \bigcup_{V_{[\underline b]_k}(A_{\mathfrak X})  \neq\emptyset} V_{[\underline b]_k}(A_0)\Bigg) \\
	  & \leq \max_{V_{[\underline b]_k}(A_{\mathfrak X})  \neq\emptyset } \text{dist}_H(V_{[\underline b]_k}(A_{\mathfrak X}), V_{[\underline b]_k}(A_0)) 
      \\ 
      & \leq \lambda^{k}\text{diam}(X)
      \end{align*}
      for every $k\ge 1$, hence proving the proposition.
\end{proof}

\begin{remark}
The previous proved that the essential part $\rm{ess}(F_{\mathfrak X}^k(A_0))$ converges exponentially to the local attractor. The sets 
$V_{[\underline b]_k}(A_0) \neq \emptyset$  for which $V_{[\underline b]_k}(A_{\mathfrak X}) = \emptyset$ appearing in ~\eqref{eq:extraterm} will disappear as $k \to \infty$ but the velocity in which it happens will strongly  depend on the play between the restrictions and the contraction rate.
\end{remark}

\section{Examples}\label{sec:examples}

In this section we provide a wide collection of examples which illustrate the main results of the paper and highlight the key differences between local attractors and the classical context of attractors for IFSs.
We start with the following example, due to F. Strobin, which illustrates that the closedness hypothesis on the domains of the maps in Theorem~\ref{thm:topological} is in fact essential.

\begin{example} ({\bf Necessity of closed restrictions}) \label{exam:filip}
    Consider the unit interval $X=[0,1]$, the and the subsets $X_1=\left(0, \frac{1}{2}\right]$ and $X_2=[0,1]$. Let $f_1: X_1 \rightarrow X$ be a continuous map such that 
$$
f_1\left(\left(0, \frac{1}{2^n}\right]\right)=\left[0, \frac{1}{2^{n+1}}\right]
$$ 
for every ${n\ge 1}$
and $ f_2: X_2 \rightarrow X$ be constant map  $f_2(x)=1$ for every $x\in X_2$. 
Notice that
$$
F_{\mathfrak X}([0,1])=f_1\left(\left(0, \frac{1}{2}\right]\right) \cup f_2([0,1])=\left[0, \frac{1}{4}\right] \cup\{1\} \
$$
and
$$
F^2_{\mathfrak X}([0,1])=f_1\left(\left(0, \frac{1}{4}\right]\right) \cup f_2\left(\left[0, \frac{1}{4}\right] \cup\{1\}\right)=\left[0, \frac{1}{8}\right] \cup\{1\}.
$$
Proceeding by induction one concludes that 
$ 
F^n_{\mathfrak X}([0,1])=\left[0, \frac{1}{2^{n+1}}\right] \cup\{1\}
$ 
for every $n\ge 1$,
hence
\begin{equation}
\label{ex:noninv1}
\bigcap_{n \ge 1} F^n_{\mathfrak X}([0,1])=\{0,1\}.    
\end{equation}
However, as $F_{\mathfrak X}(\{0,1\})=f_1(\emptyset) \cup f_2(\{0,1\})=\{1\}$ one concludes that the attractor defined by \eqref{ex:noninv1} is not $F_{\mathfrak X}$-invariant. 
\end{example}

\begin{example}\label{ex:differnt-basins}
({\bf Strict inclusion of basins of attraction})
    Consider $X=[0,4]$, the family of restrictions $X_1= X_2=X_3=[0,1]$, $X_4=\{3\}$ and the maps $f_1(x)=\frac{1}{3} x, x \in X_1$, $f_2(x)=\frac{1}{3} x + \frac{2}{3}, x \in X_2$, $f_3(x)=2, \forall x \in X_3$ and $f_4(x)=4, \forall x \in X_4$.

    Using Proposition~\ref{prop:set prop F_loc} it is not hard to check that $A_{\mathfrak X} =\bigcap_{i\geq 0} F_{\mathfrak X}^{i}(X) = \mathcal{C} \cup \{2\}$ where $\mathcal{C} \subset [0,1]$ is the middle third Cantor set. 
    Moreover, by Theorem~\ref{thm:topological}, 
        \[\emptyset \neq \mathcal{B}_{\mathfrak X}^{\, \text{inv}} \subseteq \mathcal{B}_{\mathfrak X}^{\, \text{out}} \subseteq \mathcal{B}_{\mathfrak X}\subseteq K^*(X).\]
    We claim that all previous inclusions are strict. In order to prove that, consider the sets     
    $$
    A'_0:=\mathcal{C} \cup \{2,3\}, \quad  A''_0:= \{0,2\} \quad \text{and} \quad A'''_0:=\{3\}.$$

First we note that $A'_0 \in \mathcal{B}_{\mathfrak X}^{\, \text{out}} \setminus  \mathcal{B}_{\mathfrak X}^{\, \text{inv}}$.
    Indeed,  $F_{\mathfrak X}(A'_0) = \mathcal{C} \cup \{2,4\} $ and $F_{\mathfrak X}^i(A'_0) = \mathcal{C} \cup \{2\} = A_{\mathfrak X},$ for every $i \geq 2 $, thus 
    $A'_0 \in \mathcal{B}_{\mathfrak X}^{\, \text{out}}$.
    However, $F_{\mathfrak X}(A'_0) \not\subseteq A'_0$ thus $ A'_0 \not\in \mathcal{B}_{\mathfrak X}^{\, \text{inv}}$.
Now we observe that $A''_0 \in \mathcal{B}_{\mathfrak X} \backslash  \mathcal{B}_{\mathfrak X}^{\, \text{out}}$.
In fact,  $F_{\mathfrak X}(A''_0) = \tilde F (\{0\}) \cup \{2\} $ and $F_{\mathfrak X}^i(A''_0) = \tilde F^i (\{0\}) \cup \{2\}, \; i \geq 2 $ where $\tilde F$ is the Hutchinson-Barnsley fractal operator associated to the classical contractive IFS formed by $\{f_1, f_2\}$ in $[0,1]$ whose attractor is $\mathcal{C}$. In particular $\text{dist}_H(\tilde F^i (\{0\}), \mathcal{C}) \to 0$ as $i \to \infty$. Thus
     \[\text{dist}_H(F_{\mathfrak X}^i(A''_0),A_{\mathfrak X}) =\text{dist}_H( \tilde F^i (\{0\}) \cup \{2\}, \mathcal{C} \cup \{2\}) \leq \max \left(\text{dist}_H( \tilde F^i (\{0\}) , \mathcal{C}), \text{dist}_H(  \{2\}, \{2\})\right),\]
     hence $\lim_{i \to \infty} F_{\mathfrak X}^i(A''_0) = A_{\mathfrak X}$ which shows that $A''_0 \in \mathcal{B}_{\mathfrak X}$, while it is clear that $ A''_0 \not\in \mathcal{B}_{\mathfrak X}^{\, \text{out}}$.
 Finally, we notice that $A'''_0 \in K^*(X) \setminus \mathcal{B}_{\mathfrak X}$ because
$F_{\mathfrak X}(A'''_0) = \{4\} $ and $F_{\mathfrak X}^i(A'''_0) = \emptyset, \; i \geq 2$, hence   
$ A'''_0 \not\in \mathcal{B}_{\mathfrak X}$.
\end{example}

In what follows we will construct contractive local IFSs where 
there endpoints in the local attractor with finite positive orbits, and
we provide a complete description of the code space and the extended shift-spaces associated to the local attractor.

\begin{example}\label{ex:Elismar}
({\bf Code space for a contractive local IFS so that $A^\infty_{\mathfrak X}\subsetneq A_{\mathfrak X}$})
   Consider the square $X=[0,1]^2 \subset \mathbb{R}^2$ endowed with the Euclidean metric. 
   Firstly, consider the local iterated function system IFS$_1$ in $X$ given by $f_1(x,y)=(x/3,0)$ and $f_2(x,y)=(x/3 + 2/3,0)$ with domains $X_1= [0,1/3]\times \{0\}$ and $X_2= [2/3,1]\times \{0\}$, respectively. 
   This is 
   an embedding of the ternary Cantor set $\mathcal{C}$ on $X$ and whose local attractor is the set $A_1= \mathcal{C} \times \{0\}$.
It is clear that the code space is the set $\Sigma_1= \{1,2\}^{-\mathbb{N}}$.
    Secondly, consider the local IFS, denoted by IFS$_2$, obtained from the previous one by the addition of a third map $f_3(x,y)=q=(1,1)$ defined on the domain $X_3=X_1 \cup X_2$.
It is not hard to check that the fractal operator of these two local IFSs are related by
    \[F_2^n(X)= F_1^n(X) \cup \{q\}\]
    for every $n\ge 1$. Taking the limit as $n$ tends to infinity one concludes that the local attractor $A_2$ for IFS$_2$ is
    \[
    A_2=A_1 \cup \{q\}=\mathcal{C} \cup \{q\}.
    \]
    In particular the local attractor is not contained in $\bigcup_{i=1}^3 X_i$.
    The code space of $A_2$ is a proper subset of $ \{1,2,3\}^{-\mathbb{N}}$. In fact, 
    a careful examination shows that if $(x,y) \in \mathcal{C}$  then it will be obtained via local projection of a sequence in $\Sigma_1$, and $q$ can only be reached as the projection of a sequence $\underline b= (...,\, b_{-2}, b_{-1})$ where $b_{-1} =3$ and $(...,\, b_{-3}, b_{-2}) \in \Sigma_1$. Thus,
    \[\Sigma_2= \{1,2\}^{-\mathbb{N}} \cup (\{1,2\}^{-\mathbb{N}}\ast 3).\]
    
    \smallskip
    Finally, we define a third local IFS, denoted by IFS$_3$ and obtained from the previous one by addition of the map $f_4(x,y)=p=(0,1)$ defined on the set $X_4= X_1 \cup X_2\cup X_3\cup  \{q\} =X_1 \cup X_2 \cup  \{q\}$.
    The fractal operator of IFS$_2$  and IFS$_2$ are related by
    $F_3^n(X)= F_2^n(X) \cup \{p\}$ 
    for every $n\ge 1$, hence characterizing 
    the local attractor $A_3$ of the IFS$_3$ as
    \[A_3=A_2 \cup \{p\}=\mathcal{C} \cup \{q,p\}.\]
    The new code space for $A_3$ is a proper subset of $ \{1,2,3,4\}^{-\mathbb{N}}$. Indeed:
    \begin{itemize}
        \item every $(x,y) \in \mathcal{C}$  is only obtained via local projection of a sequence in $\Sigma_1$;
        \item $q$ can only be obtained    a sequence $(...,\, b_{-2}, b_{-1})$ where $b_{-1} =3$ and $(...,\, b_{-3}, b_{-2}) \in \Sigma_1$, meaning that  
    $q=\lim_n f_{3}(f_{b_{-3}}(f_{b_{-4}}(\cdots)))$;
        \item  $p$ can be obtained in two ways:
        \begin{itemize}
            \item[$\circ$] projection of a sequence $(...,\, b_{-2}, b_{-1})$ where $b_{-1} =4$, $(...,\, b_{-3}, b_{-2}) \in \Sigma_1$ and
    $p=\lim_n f_{4}(f_{b_{-2}}(f_{b_{-3}}(f_{b_{-4}}(\cdots))))$;
    \item[$\circ$]  projection of a sequence $(...,\, b_{-2}, b_{-1})$ where $b_{-2} =3$ and $b_{-1} =4$ and $(...,\, b_{-4}, b_{-3}) \in \Sigma_1$, that is 
    $p=\lim_n f_{4}(f_{3}(f_{b_{-3}}(f_{b_{-4}}(\cdots))))$.
        \end{itemize}
    \end{itemize}
    This shows that
    \[\Sigma_3= \{1,2\}^{-\mathbb{N}} \cup (\{1,2\}^{-\mathbb{N}}\ast 3)\cup (\{1,2\}^{-\mathbb{N}}\ast 4)\cup (\{1,2\}^{-\mathbb{N}}\ast 3\ast 4).\]
Furthermore, the extended space of orbits $\Omega^{e}(\mathfrak X)$ is formed by orbits of four types:
\begin{itemize}
    \item[(I)] Infinite orbits $((\ldots, (x_{-2},a_{-2}), (x_{-1},a_{-1}) )| ((x_0,a_0), (x_1,a_1),(x_{2},a_{2}), \ldots))$, where  \\$(\ldots,a_{-2},a_{-1}) \in  \{1,2\}^{-\mathbb{N}}$;
    \item[(II)] Orbits with end point of length one $((\ldots, (x_{-2},a_{-2}), (x_{-1},4) )| ((p,0), (p,0),(p,0), \ldots))$ for $(\ldots,a_{-2},a_{-1}) \in  \{1,2\}^{-\mathbb{N}}\ast 4$. 
    \item[(III)] Orbits with end point of length one 
     $((\ldots, (x_{-3},a_{-3}),(x_{-2},3), (q,4) )| ((p,0), (p,0),(p,0), \ldots))$ for $(\ldots,a_{-3},{3},4) \in  \{1,2\}^{-\mathbb{N}}\ast 3\ast 4$.
    \item[(IV)] Orbits with end point of length two  $((\ldots,(x_{-2},a_{-2}), (x_{-1},3))| ( (q,4), (p,0), (p,0),(p,0), \ldots))$ for $(\ldots,a_{-2},a_{-1}) \in  \{1,2\}^{-\mathbb{N}} \ast 3$.
\end{itemize}
In this example, the set $\Omega$ formed by orbits that always remain in the union of the domains of the maps $f_i$ (recall ~\eqref{eq:def-enlarged-Omega} and Theorem~\ref{thmC}) coincides with the $\pi^{+}$ projection of all type (I) orbits.
\end{example}

We observe that for the local attractor in Example~\ref{ex:Elismar} one has $A_{\mathfrak X}^{\infty} = \mathcal{C}$ is the middle third Cantor set, that $A_{\mathfrak X} \setminus A_{\mathfrak X}^{\infty} = \{q,p\}$, and  $\text{dist}_H(A_{\mathfrak X} \setminus A_{\mathfrak X}^{\infty}, A_{\mathfrak X}^{\infty}) > 0$.
The next example shows that, in general, one cannot expect to separate the set of points $A^\infty_{\mathfrak X}$ in the local attractor that have infinite orbits from the set
$A_{\mathfrak X}\setminus A^\infty_{\mathfrak X}$ of endpoints of the attractor. 

\begin{example} 
    \label{ex:shift} 
    ({\bf A contractive local IFS so that  $ A_{\mathfrak X}\setminus A^\infty_{\mathfrak X}$ accumulates on $A^\infty_{\mathfrak X}$}) 
 Consider the metric space $X = \{0,1,2\}^{\mathbb{N}}$ endowed with the product topology. Define the following maps and their corresponding domains:
\[
\begin{aligned}
f_0(x) &:= (0, x_1, x_2, \ldots), & X_0 &:= \{0,1\}^{\mathbb{N}}, \\
f_1(x) &:= (1, x_1, x_2, \ldots), & X_1 &:= \{0,1\}^{\mathbb{N}}, \\
f_2(x) &:= (0,0, x_1+x_2, x_2+x_3, \ldots), & X_2 &:= X_0 \cup X_1 = \{0,1\}^{\mathbb{N}}.
\end{aligned}
\]
It is clear from the definition that $\operatorname{Lip}(f_0) = \operatorname{Lip}(f_1) = 1/2$ and that
\begin{equation}
    \label{eq:ping}
    f_0(\{0,1\}^{\mathbb{N}}) \cup f_1(\{0,1\}^{\mathbb{N}}) = \{0,1\}^{\mathbb{N}}.
\end{equation}
We claim that $\operatorname{Lip}(f_2) \le 1/2$.
Indeed, for $x, y \in \{0,1\}^{\mathbb{N}}$ with $d(x,y) = e^{-N}$, one can write
\[
x = (a_1, a_2, \ldots, a_N, x_{N+1}, \ldots), \quad 
y = (a_1, a_2, \ldots, a_N, y_{N+1}, \ldots), \quad x_{N+1} \neq y_{N+1}
\]
and observe that 
\[
\begin{aligned}
d(f_2(x), f_2(y)) &= d\big((0,0,a_1+a_2, \ldots, a_N + x_{N+1}, \ldots), (0,0,a_1+a_2, \ldots, a_N + y_{N+1}, \ldots)\big) \\
&\le e^{-N-1} = \frac{1}{2} d(x,y),
\end{aligned}
\]
since $a_N + x_{N+1} \neq a_N + y_{N+1}$.
Moreover, as 
$ 
f_2(0,1,1,\ldots) = (0,0,1,2,\ldots) \notin \{0,1\}^{\mathbb{N}}.
$ 
one obtains that 
\[
f_2(\{0,1\}^{\mathbb{N}}) \not\subset \{0,1\}^{\mathbb{N}}.
\] 
In particular, $R_{\mathfrak X} = (X_j, f_j)_{1\le j \le 3}$ is a contractive local IFS. 

\medskip
Let $A_{\mathfrak X}$ denote the local attractor of $R_{\mathfrak X}$ and $A_{\mathfrak X}^{\infty}$ be as in ~\eqref{eq:maxinvA}. We claim that 
\[
A_{\mathfrak X} \setminus A_{\mathfrak X}^{\infty} \neq \emptyset \quad \text{and} \quad \text{dist}_H(A_{\mathfrak X} \setminus A_{\mathfrak X}^{\infty}, A_{\mathfrak X}^{\infty}) = 0.
\] 
We start by noticing that
\begin{equation}
\label{eq:preciseinc1}
\{0,1\}^{\mathbb{N}} \subset A_{\mathfrak X}^{\infty}
\quad \text{and} \quad 
A_{\mathfrak X}=\{0,1\}^{\mathbb{N}} \cup Q,   
\end{equation} 
where $Q := f_2(\{0,1\}^{\mathbb{N}})$. 
In fact, $\{0,1\}^{\mathbb{N}} \subset A_{\mathfrak X}$ as a consequence of ~\eqref{eq:ping}.   
Moreover, as every point $x \in \{0,1\}^{\mathbb{N}}$ has an infinite orbit one concludes that $\{0,1\}^{\mathbb{N}} \subset A_{\mathfrak X}^{\infty}$, thus proving the  inclusion 
in \eqref{eq:preciseinc1}. Now, taking the image of the set $\{0,1\}^{\mathbb{N}} \cup Q$ by the Hutchinson-Barnsley operator acting we obtain
\[
\begin{aligned}
F_{\mathfrak X}(\{0,1\}^{\mathbb{N}} \cup Q) &= \bigcup_{j=0}^2 f_j\big((\{0,1\}^{\mathbb{N}} \cup Q) \cap X_j\big) \\
&= \{0,1\}^{\mathbb{N}} \cup f_2(\{0,1\}^{\mathbb{N}}) 
= \{0,1\}^{\mathbb{N}} \cup Q.
\end{aligned}
\]
This proves that $\{0,1\}^{\mathbb{N}} \cup Q$ is positively $F_{\mathfrak X}$-invariant. By maximality of the local attractor (recall Theorem~\ref{thm:topological}) we conclude that $\{0,1\}^{\mathbb{N}} \cup Q \subseteq A_{\mathfrak X}$.
The converse inclusion follows from the fact that every point in $Q$ is accumulated by orbits with arbitrary length, under compositions of maps of the form $f_2 \circ f_{j_k}\circ \dots \circ f_{j_1}$ with $j_\ell\in\{0,1\}$ for every $1\le \ell \le k$ (recall Proposition~\ref{prop: infinite orbits attractor}).
This completes the proof of the claim in \eqref{eq:preciseinc1}. 

\smallskip
A closer examination of the action of the map $f_2$ shows that 
there exists a partition $\{0,1\}^{\mathbb{N}} = \Sigma_M \sqcup (\{0,1\}^{\mathbb{N}} \setminus \Sigma_M)$, where $\Sigma_M \subset \{0,1\}^{\mathbb{N}}$ is the sub-shift of finite type determined by the forbidden word $11$,  in such a way that $f_2(\Sigma_M) \subset \{0,1\}^{\mathbb{N}}$  and $f_2(\{0,1\}^{\mathbb{N}} \setminus \Sigma_M)= Q \setminus \{0,1\}^{\mathbb{N}}$.
Therefore, 
\begin{equation}
    \label{eq:decomp-attractor}
A_{\mathfrak X} \setminus A_{\mathfrak X}^{\infty}= f_2(\{0,1\}^{\mathbb{N}} \setminus \Sigma_M) \quad\text{and}\quad 
A_{\mathfrak X} = \{0,1\}^{\mathbb{N}} \sqcup f_2(\{0,1\}^{\mathbb{N}} \setminus \Sigma_M).
\end{equation}

We proceed to show that $ A_{\mathfrak X}\setminus A^\infty_{\mathfrak X}$ accumulates on $A^\infty_{\mathfrak X}$.
For each $n\ge 1$ consider the sequence
\[
\underline b^n := (0, \ldots, 0, 1, 1, 0, 0, \ldots) = (0^n, 1,1,0^\infty) \in \{0,1\}^{\mathbb N}.
\]
It is clear that $\underline b^n \in A_{\mathfrak X}^{\infty}$  
and
\[
f_2(\underline b^n) = (0^{n+1}, 1, 2, 1, 0^\infty) \in Q \subset A_{\mathfrak X}, 
\quad f_2(\underline b^n) \notin X_0 \cup X_1 \cup X_2,
\]
hence $f_2(\underline b^n) \in A_{\mathfrak X} \setminus A_{\mathfrak X}^{\infty}$. However, as $d(\underline b^n, f_2(\underline b^n)) = e^{-n}$  tends to zero as $n$ tends to infinity, this proves that $\text{dist}_H(A_{\mathfrak X} \setminus A_{\mathfrak X}^{\infty}, A_{\mathfrak X}^{\infty}) = 0$.
\end{example}

\begin{example}\label{ex:nonSFT}
({\bf Contractive local IFS whose code space is not an SFT}) 
Let $R_{\mathfrak X}=(X_j,f_j)_{0\le j \le 2}$
    be the contractive local IFS described in Example~\ref{ex:shift}, whose attractor $A_{\mathfrak X}$ was described in  ~\eqref{eq:decomp-attractor}. We claim that the code space $\Sigma_{\mathfrak X} \subset \{0,1,2\}^{-\mathbb{N}}$ has a more complex structure and that it is not a subshift of finite type. One knows that 
$\Sigma_{\mathfrak X}\subset \{0,1,2\}^{-\mathbb{N}}$ is a subshift of finite type
if and only if  $\sigma^-: \Sigma_{\mathfrak X} \to \Sigma_{\mathfrak X}$ is an open map (see e.g. \cite[Theorem~3.35]{Kur03}).

\smallskip
Given the open set $U:=[2] = \{\underline{b} \in \Sigma_{\mathfrak X} \colon b_{-1}=2\} \subset \Sigma_{\mathfrak X}$ we claim that its image 
\[
V:=\sigma^-(U) =\{\underline{c} \in \Sigma_{\mathfrak X} | \underline{c}*2 \in \Sigma_{\mathfrak X}\} \subset \Sigma_{\mathfrak X}
\]
is not open. In order to prove this fact, consider the sequence
$
\underline{b}=(\ldots, 1,0\ldots,1,0,2,2)
$ 
formed by infinitely many blocks $1,0$ and ending in $2,2$ and note that $\underline{b} \in U$ as
\begin{align*}
\pi_{{\mathfrak X}}(\underline{b})
& =\lim_{k\to \infty} f_2^2\circ (f_0\circ f_1)^{k}(X) = f_2^2 (\lim_{k\to \infty}(f_0\circ f_1)^{k}(X)) =  f_2^2 (0,1,0,1,0,1,\ldots)    \\
&= f_2 (0,0,1,1,1,1,1,\ldots) = (0,0, 0,1,2,2,2,2,....) \in A_{\mathfrak X}.
\end{align*}
By the invariance of $\Sigma_{\mathfrak X}$, it follows that $\underline{c}=\sigma^-(\underline{b})= (\ldots, 1,0\ldots,1,0,2) \in V \subset \Sigma_{\mathfrak X}$.

We claim that $\underline{c}$ is not an interior point of $V$.
Indeed, for each $n\ge 1$ take the sequence 
\[
\underline c^m := (\ldots, 0,1,0, 1,0,1,1, \boxed{1,0\ldots,1,0},2) \in \{0,1,2\}^{-\mathbb{N}},
\]
where $m \geq 1$ stands for the number of 
blocks of the form $1,0$ between the block $1,1$ and the final symbol $2$.
It is clear that $d(\underline c^m,\underline{c})=e^{-(2m+1)}$ tends to zero as $m \to \infty$, but $\underline c^m$ does not belong to $V$, because
\begin{align*}
\pi_{{\mathfrak X}}(\underline c^m)
& =f_2\circ (f_0\circ f_1)^{m} \circ f_1^2 (\lim_{k\to \infty}(f_0\circ f_1)^{k}(X)) 
\\
&=  f_2\circ (f_0\circ f_1)^{m}   (1,1,0,1,0,1,0,1,\ldots) \\
& 
=(0,0,1,1,\ldots, 1,2,1,1,1,1,\ldots)  \in \Sigma_{\mathfrak X}
\end{align*}
does not belong to $V$ because one can not apply $f_2$ to this point.
\end{example}

\begin{example}({\bf Local attractors with non-constant local dimension})
\label{ex:MapleSierpinski}
Consider the contractive local IFS 
$R_{\mathfrak X}=(X_i, f_i)_{1\le i \le 9}$ where the maps are defined by 
\[\begin{cases}
	 	f_1(x,y)= (0.8 x+0.1, 0.8 y+0.04)\\
	 	f_2(x,y)=  (0.5 x+0.25, 0.5 y+0.4 )\\
	 	f_3(x,y)=  ( 0.355 x-0.355 y+0.266, 0.355 x+0.355 y+0.078)\\
	 	f_4(x,y)= ( 0.355 x+0.355 y+0.378, -0.355 x+0.355 y+0.434)\\
        f_5(x,y)= (1+ ((x-1)+0)/2,1+ ((y-1)+0)/2)\\  
        f_6(x,y)=(1+ ((x-1) +1)/2,1+ ((y-1)+0)/2)\\
        f_7(x,y)=(1+ ((x-1) +1/2)/2,1+ ((y-1) +(\sqrt(3)/2))/2)\\
        f_8(x,y)=(\gamma x+ p_1,\gamma y+ p_2)\\
         f_9(x,y)= (\beta x+ q_1,\beta y+ q_2),
	 \end{cases}
\]
where $p_1=0.5, p_2=1.5, q_1=1.2, q_2=0.2, \gamma=0.3, \; \beta=0.6$  
and the sets $X_i\subset [0,2]^2$  are  given by $X_1=X_2=X_3=X_4=[0,1]^2$, $X_5=X_6=X_7=[0,2]\times [1,2]$ and $X_8=X_9=[0,1]^2$.

We note that the local IFS  $(X_i,f_i)_{1\le i \le 4}$ generates the classical fractal Maple leaf in $[0,1]^2$ and the local IFS  $(X_i,f_i)_{5\le i \le 7}$  generates a classical Sierpinski triangle in up-right corner $[1,2]^2$. As $[0,2]\times [1,2] \supset[0,1]^2$ some copies of the Maple leaf will accumulate at the Sierpinski via the map $f_8$, as this is a similarity around the point $p=(p_1,p_2)=(0.5,1.5)$. The map $f_9$, being a similarity around the point $q=(q_1,q_2)=(1.2,0.2)$, generates a copy in small scale of the Maple leaf in the bottom-left corner (grey) entirely formed by endpoints (since no domain will intercept it).

A numerical approximation of the local attractor $A_{\mathfrak X}$ is depicted in Figure~\ref{fig:maple sierp merge}. 
\begin{figure}[htb]
    \centering
\includegraphics[width=0.3\linewidth]{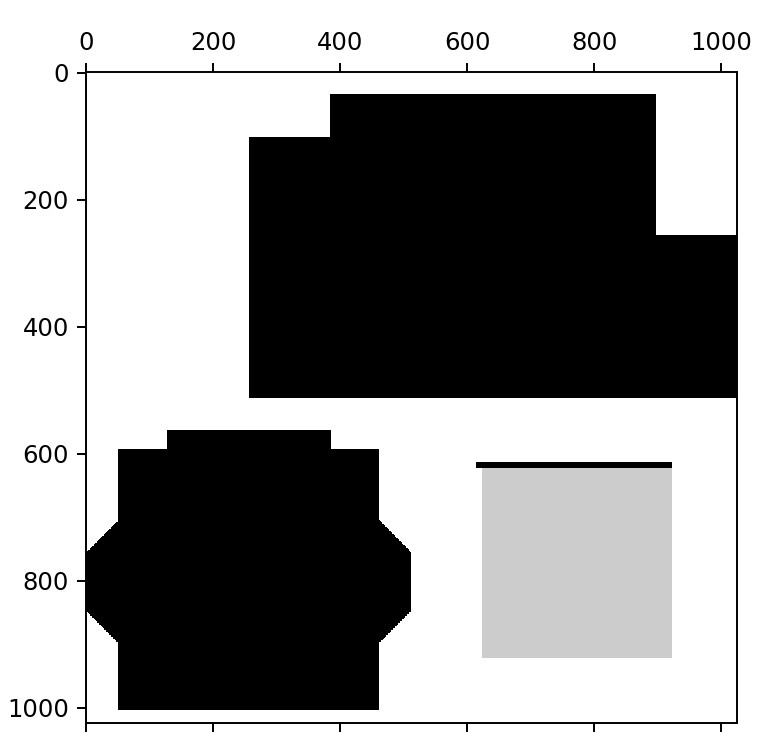} \quad \includegraphics[width=0.3\linewidth]{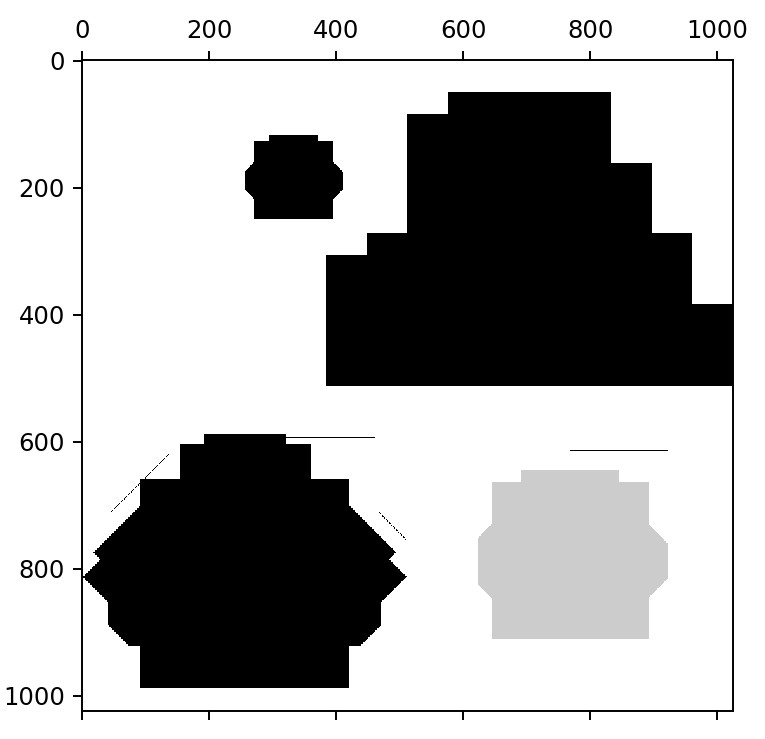}\\    \includegraphics[width=0.3\linewidth]{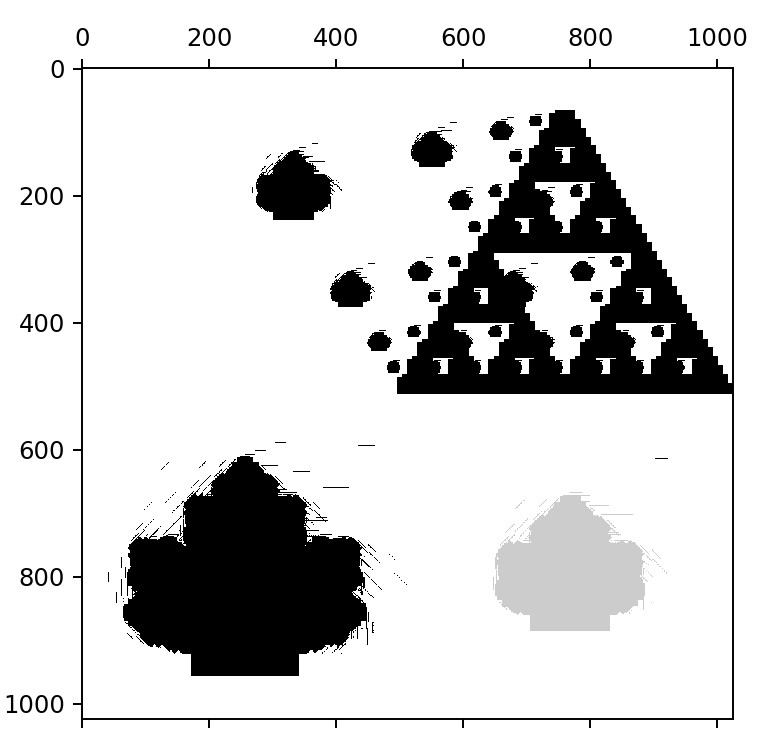} \quad \includegraphics[width=0.3\linewidth]{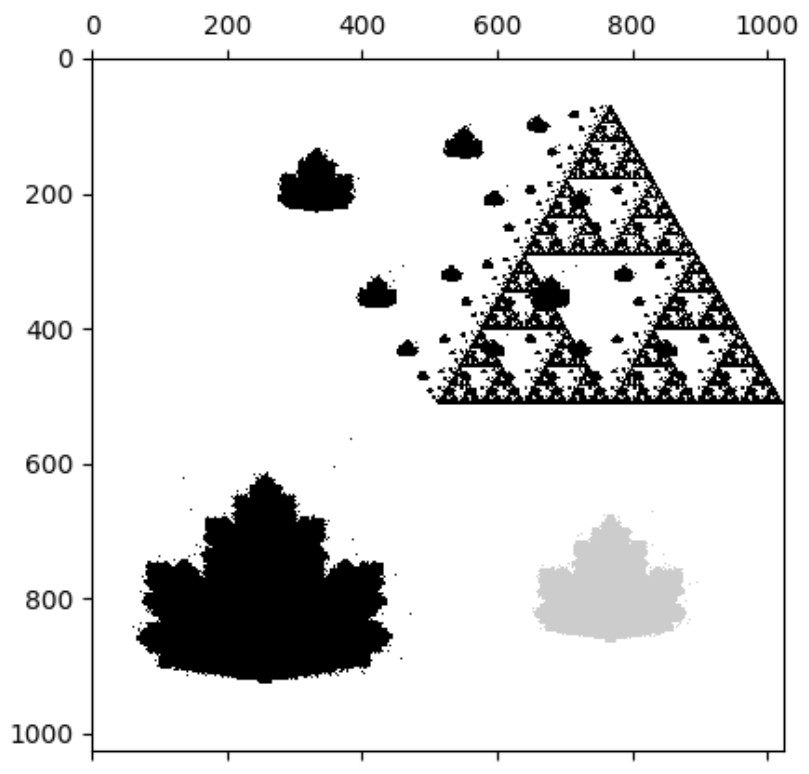}
    \caption{Iterates $F_{\mathfrak X}^1(X),F_{\mathfrak X}^2(X), F_{\mathfrak X}^5(X)$ and $F_{\mathfrak X}^{15}(X)$ (left to right, up to bottom).}
    \label{fig:maple sierp merge}
\end{figure}
The local attractor is highly non self-similar because the Maple leaf and the Sierpinski have different local structures and different local Hausdorff dimension.
Finally, the attractor $A_{\mathfrak X}$ is non-transitive, as there exists a set of endpoints at positive distance of the remaining of the local attractor. 
\end{example}

\bigskip
\subsection*{Acknowledgments} 

We are grateful to F. Strobin his comments on a preli\-minary version of the text and for providing Example~\ref{exam:filip}. 
   This work was initiated during a visit of the first named author to University 
of Aveiro. ERO was partially supported by MATH-AMSUD under the project GSA/CAPES, Grant 88881.694479/2022-01, and by CNPq Grant 408180/2023-4.

PV was partially supported by CIDMA under the
Portuguese Foundation for Science and Technology 
(FCT, https://ror.org/00snfqn58)  
Multi-Annual Financing Program for R\&D Units,
grants UID/4106/2025 and UID/PRR/4106/2025.

\end{document}